\documentclass{article}
\usepackage{amsmath,amssymb,amsthm}
\newtheorem{Theorem}{Theorem}[section]
\newtheorem{Proposition}[Theorem]{Proposition}
\newtheorem{Lemma}[Theorem]{Lemma}
\newtheorem{Corollary}[Theorem]{Corollary}
\newtheorem{Definition}{Definition}[section]
\newtheorem{Remark}{Remark}[section]

\usepackage[colorlinks=true,linkcolor=blue,citecolor=green]{hyperref}

\usepackage[margin=4cm]{geometry}
\newlength{\bibitemsep}
\newlength{\bibparskip}
\let\oldthebibliography\thebibliography
\renewcommand\thebibliography[1]{%
  \oldthebibliography{#1}%
  \setlength{\parskip}{\bibitemsep}%
  \setlength{\itemsep}{\bibparskip}%
}

\numberwithin{equation}{section}

\begin{document}

\title{\Large\bf The gyrokinetic limit for the two dimensional Vlasov-Poisson system with multi-point charges}
\author{Jingpeng Wu\thanks{Department of Applied Mathematics, Nanjing Forestry University, Nanjing, 210037, China ({\protect e-mail:\url{jp.wu@njfu.edu.cn}}).}}
\date{}
\maketitle

\begin{abstract}
In this article, we investigate the gyrokinetic limit for the two dimensional Vlasov-Poisson system with multi-point charges. We show that the solution converges to a measure-valued solution of the Euler equation with a defect measure, which extends the results of Miot (Nonlinearity 32(2):654-677, 2019) to the case of multi-point charges and removes the smallness condition $\sup_{0<\varepsilon<1}\|f_{\varepsilon}^0\|_{L^1}<1$. The main difficulty arises from the fact that the orbits of point charges will intersect frequently and rapidly as the magnetic field intensity becomes large. To overcome the problem, we adopt the techniques recently developed by Wu and Zhang (Journal of Statistical Physics 190:183, 2023) combined with a new technique introduced here.


\noindent{\bf MR Subject Classification:} 35Q83; 82D10; 35B40; 35A24.
\end{abstract}

\section{Introduction and main results}

In this paper, we investigate the gyrokinetic limit for the  Vlasov-Poisson system with $N$ point charges, so called the Plasma-Charge model, which is described by the asymptotic behaviour of solutions to the following equations as $\varepsilon$ tends to zero:
\begin{equation}\label{eq-VP-point}\tag{$\mathbf{E}^{\varepsilon}$}
\left\{
\begin{split}
&\partial_tf_{\varepsilon}+\frac{v}{\varepsilon}\cdot\nabla_xf_{\varepsilon}+\left(\frac{v^{\perp}}{\varepsilon^2}+\frac{E_{\varepsilon}+\gamma F_{\varepsilon}}{\varepsilon}\right)\cdot\nabla_vf_{\varepsilon}=0,\\
&E_{\varepsilon}(t,x)=\int_{\mathbb{R}^2}\frac{x-y}{\vert x-y\vert^2}\rho_{\varepsilon}(t,y)\,\mathrm{d}y,\quad\rho_{\varepsilon}=\int_{\mathbb{R}^2}f_{\varepsilon}\,\mathrm{d}v,\\
&\dot{\xi}_{\alpha,\varepsilon}(t)=\frac{\eta_{\alpha,\varepsilon}(t)}{\varepsilon},\quad\dot{\eta}_{\alpha,\varepsilon}(t)=\gamma\left(\frac{\eta_{\alpha,\varepsilon}^{\perp}(t)}{\varepsilon^2}+\frac{E_{\varepsilon}(t,\xi_{\alpha,\varepsilon})}{\varepsilon}\right)+\frac{1}{\varepsilon}F_{\varepsilon}(t,\xi_{\alpha,\varepsilon}(t)),\\
&(f_{\varepsilon},\xi_{\alpha,\varepsilon},\eta_{\alpha,\varepsilon})\vert_{t=0}=(f_{\varepsilon}^{0},\xi_{\alpha,\varepsilon}^0,\eta_{\alpha,\varepsilon}^0),\quad\alpha=1,2,\dots,N.
\end{split}\right.
\end{equation}
with $F_{\varepsilon}$ defined by
\begin{equation*}
F_{\varepsilon}(t,x)=\left\{\begin{split}
\sum_{\alpha}\frac{x-\xi_{\alpha,\varepsilon}(t)}{\vert x-\xi_{\alpha,\varepsilon}(t)\vert^2},\quad&\text{for}\,\,x\notin\{\xi_{\alpha,\varepsilon}(t)\},\\
\sum_{\beta:\beta\ne\alpha}\frac{\xi_{\alpha,\varepsilon}(t)-\xi_{\beta,\varepsilon}(t)}{\vert\xi_{\alpha,\varepsilon}(t)-\xi_{\beta,\varepsilon}(t)\vert^2},\quad&\text{for}\,\,x=\xi_{\alpha,\varepsilon}(t).
\end{split}\right.
\end{equation*}
Here $f_{\varepsilon}=f_{\varepsilon}(t,x,v)$ is the  phase space density of the plasma particles at time $t\ge 0$, located at $x\in\mathbb{R}^2$  with  velocity  $v\in\mathbb{R}^2$. The $\alpha$-th point charge is located at  $\xi_{\alpha,\varepsilon}(t)$ with velocity $\eta_{\alpha,\varepsilon}(t)$. All the particles are submitted to the self-consistent electric field $E_{\varepsilon}(t,x)$, to the singular field $F_{\varepsilon}(t,x)$ induced by the point charges and to a large external  magnetic field, orthogonal to the plane, which is represented by the terms  $v^{\perp}/\varepsilon^2$ or  $\eta_{\alpha,\varepsilon}^{\perp}/\varepsilon^2$, here $(v_{1},v_{2})^{\perp}=(-v_{2},v_{1})$. The real number $\gamma\gtrless 0$ corresponds to a repulsive or attractive interaction between the plasma and point charges. In this paper, we consider the repulsive case and set $\gamma=1$ for convenience.

Roughly speaking, when $\varepsilon\to 0$, the solutions of the system \eqref{eq-VP-point} converges to the measure-valued solutions of the following Euler equation with a defect measure $\nu \in L^{\infty}_t\mathcal{M}_x$:
\begin{equation}\label{eq-mv-E-defect}\tag{$\mathbf{E}_{\nu}^0$}
\left\{
\begin{split}
&\partial_t(\rho+\bar{\delta})+E_{\rho+\bar{\delta}}^{\perp}\cdot\nabla(\rho+\bar{\delta})=\nabla^2:\nu,\\
&E_{\rho+\bar{\delta}}=\frac{x}{\vert x\vert^2}*(\rho+\bar{\delta}),\quad\bar{\delta}(t,x)=\sum_{\alpha}\delta_{\xi_{\alpha}(t)}(x),
\end{split}\right.
\end{equation}
and if the defect measure $\nu$ is vanishing, \eqref{eq-mv-E-defect} reduces further to the vortex-wave system:
\begin{equation}\label{eq-Vortex-Wave}\tag{$\mathbf{E}_0^0$}
\left\{
\begin{split}
&\partial_t\rho+(E_{\rho}+F)^{\perp}\cdot\nabla\rho=0,\quad E_{\rho}=\frac{x}{\vert x\vert^2}*\rho,\\
&\dot{\xi}_{\alpha}(t)=\big(E_{\rho}(t,\xi_{\alpha}(t))+F(t,\xi_{\alpha}(t))\big)^{\perp},
\end{split}\right.
\end{equation}
with $F$ defined by
\begin{equation*}
F(t,x)=\left\{\begin{split}
\sum_{\alpha}\frac{x-\xi_{\alpha}(t)}{\vert x-\xi_{\alpha}(t)\vert^2},\quad&\text{for}\,\,x\notin\{\xi_{\alpha}(t)\},\\
\sum_{\beta:\beta\ne\alpha}\frac{\xi_{\alpha}(t)-\xi_{\beta}(t)}{\vert\xi_{\alpha}(t)-\xi_{\beta}(t)\vert^2},\quad&\text{for}\,\,x=\xi_{\alpha}(t).
\end{split}\right.
\end{equation*}

\subsection{Background}

The classical Vlasov-Poisson system, \eqref{eq-VP-point} without charge nor magnetic field, has been widely studied, see e.g., \cite{Ars73} for weak solutions, \cite{UO78} for classical solutions in two dimensions, \cite{Pfa92,LP91} for classical solutions in three dimensions, \cite{Loe06,Mio16CMP} for the uniqueness of weak solutions. 

The Plasma-Charge model, \eqref{eq-VP-point} without magnetic field was first introduced by Caprino and Marchioro \cite{CM10}, in which the existence and uniqueness of the classical solution have been established  for the repulsive case $\gamma>0$ and for initial data being vacuum around point charges and satisfying some integrable conditions as in Theorem~\ref{thm-main} below. For three dimensional case, the well known Lagrangian method of \cite{Pfa92} and the Eulerian method of \cite{LP91} were developed in \cite{MMP11} and \cite{DMS15} respectively. The analysis in \cite{CM10,MMP11} relies on the bound of a pointwise energy function defined as $\bar{h}=\frac{1}{2}\vert v-\eta_{\alpha}\vert^2+\gamma G_d(x-\xi_{\alpha})$ with $G_d$ the Green function of $-\Delta$ in $\mathbb{R}^d$, which is able to handle the strictly positive distance between the plasma and the point charges, and the singular field $F_{\varepsilon}$ is always regular. Such mechanism is not applicable to the attractive case $\gamma<0$, which makes it more difficult to study. Nevertheless, the existence of classical solutions in $\mathbb{R}^2$ for $\gamma<0$ was proved in \cite{CMMP12}, based on some special properties of the logarithmic potential $G_2(x)=-\frac{1}{2\pi}\ln\vert x\vert$.

When adding an external magnetic field, the existence and uniqueness of classical solutions  to \eqref{eq-VP-point} for $\gamma>0$, might be proved by adapting the methods in \cite{UO78,Pfa92,CM10,MMP11} in dimension $2$ and $3$. When considering the asymptotic behaviour of solutions to \eqref{eq-VP-point} as $\varepsilon\to 0$, the case without point charge has been studied by Golse and Saint-Raymond \cite{GS99}, which established the convergence of \eqref{eq-VP-point} to the incompressible Euler equation with a defect measure \eqref{eq-mv-E-defect} (see also Brenier's work \cite{Bre00} in a different scaling). Later, Saint-Raymond \cite{Sai02} proved that the defect measure $\nu$ vanishes for sufficiently regular initial data. The same kind of convergence results as in \cite{GS99,Sai02} are obtained in \cite{Mio16} thanks to different techniques. Besides the already mentioned articles, there is a wide literature devoted to various asymptotical regimes for linear or nonlinear Vlasov-like equations with strong external magnetic field, leading to different nonlinear equations (see e.g., \cite{FS98,FS00,Sai00,FS01,GS03,BOS09,FR16}).

Recently Miot \cite{Mio19} investigated the gyrokinetic limit of \eqref{eq-VP-point} with a point charge for $\gamma>0$ and proved that the solution converges to a measure-valued  solution of the Euler equation with a defect measure \eqref{eq-mv-E-defect}. Furthermore, the limiting equation exactly yields the vortex-wave system \eqref{eq-Vortex-Wave} if the defect measure vanishes and under more regularity assumptions on $\rho, \xi$. Recalling that the vortex-wave system \eqref{eq-Vortex-Wave} was introduced by Marchioro and Pulvirenti \cite{MP91} and there are quite rich research literatures, which are not listed here.

\subsection{Contributions of the paper}

The aim of this article is to extend the results in \cite{Mio19} to the system \eqref{eq-VP-point} with $N>1$. Such extension seems non-trivial. Indeed, there are many works that have been only proven for the case with a single point charge, see e.g., \cite{DMS15,XZ21,AW21,PW21,PWY22,HW22}. To the best of the author's knowledge, it is unknown whether these results can be extended to the case with multi-point charges, and the difficulties posed by multi-point charges are varied. In the fluid theory, the generalization of the study on the fluid models coupled with single object to multi objects also seems to present complex difficulties, see e.g., for the vortex-wave system, the extension of \cite{LM09} to \cite{LM21}, and for the fluid-rigid body system, \cite{GLS14,GLS16,GMS18} to \cite{GS19}.

Our extension of \cite{Mio19} is based on two techniques. The first technique is introduced in \cite{WZ23}, which is used to extend the moment propagation of \cite{DMS15} to \eqref{eq-VP-point} with $N>1$ in three dimensions. The main idea in \cite{WZ23} is to find a suitable pointwise energy functional, whose derivative does not contain the principal singular terms, and analogous thought has already manifested in \cite{LM21} for the analysis of the vortex-wave system. The approach in \cite{WZ23} can be applied to estimate the time integral of the electric field along the trajectories of the point charges by more careful calculation, which allows us to remove the condition $\sup_{0<\varepsilon<1}\|f_{\varepsilon}^0\|_{L^1}<1$ in \cite{Mio19}, see Sec.~\ref{sec-est-point-singular-field}. 

The second technique is introduced in this paper for the proof of the continuity or almost continuity of the trajectories of the point charges in Sec.~\ref{sec-time-regularity}. Firstly, thanks to the new estimates established in Corollary~\ref{coro-key-L0}, we reprove the weak formula in Proposition~\ref{prn-weak-form} with an improved estimate on the remainder term. Secondly, the strong magnetic field causes the orbits of point charges to cross each other rapidly, which brings a problem different from the case $N=1$. In order to describe the problem more clearly, we consider a simple model in Sec.~\ref{sec-2} and give a detailed explanation in Remark~\ref{rem-sec-2-rem-2-2}. Hence the contradiction argument in \cite[\S.2.5]{Mio19} is not enough to obtain the almost time continuity of the trajectory of each point charge, which can only restore the time continuity of the trajectory of the center of the point charges, i.e., $\vert\sum_{\alpha}\xi_{\alpha,\varepsilon}(t)-\sum_{\beta}\xi_{\beta,\varepsilon}(s)\vert\lesssim (t-s)^{1/4}$, see Lemma~\ref{lem-min-continuity-point-charge} and Remark~\ref{rem-4-1}. To overcome this problem, we observe that within a small and non-uniform time period, the continuity of each point charge can be deduced from Lemma~\ref{lem-min-continuity-point-charge}. Note the later is global in time, hence it is able to iterate this process with a non-uniform time period $O(\varepsilon^{m^{-k}})$ for some $m>1$, which converges to $O(1)$ as the number of iteration steps $k\to\infty$, to prove finally the continuity for each point charge on small but uniform time interval, see Lemma~\ref{lem-small-time-continuity-point-charge}.

Besides, due to the improved estimates in Proposition~\ref{prn-continuity-point-charge}, the compactness argument in \cite{Mio19} can be simplified to some extent, without using the discontinuous version of Arzel\`a-Ascoli theorem, see Sec.~\ref{sec-proof-of-main-thm}.

\subsection{Notations and Definitions}

\begin{itemize}
\item When without specifying, C always denotes a positive constant depending only on the constants $N,T,K_0,K_1$ given in Theorem~\ref{thm-main}.
\item For $A,B\in\mathbb{R}^{m\times n}$, $A:B=\sum_{1\le i\le m,\,1\le j\le n} A_{ij}B_{ij}$.
\item $\delta_{z}$ denotes the Dirac measure on $\mathbb{R}^2$, concentrated at $\{z\}$. 
\item A space of functions from domain $X$ to range $Y$ denotes as the form $L(X,Y)$, or $L(X)$ if $Y=\mathbb{R}$, or $L_+(X)$ if $Y=\mathbb{R}_+$. The norms of Lebesgue spaces $L^p(\Omega)$ denote $\|\cdot\|_{L^p(\Omega)}$ or $\|\cdot\|_{p}$ when $\Omega$ is clear from the context. $\mathcal{M}$ denotes the space of bounded real Radon measures. $C_0$ denotes the completion of $C_c$, the space of continuous functions with compact supports, with respect to $\|\cdot\|_{\infty}$. We always equip the vague topology i.e.~weak-* topology $\sigma(\mathcal{M},C_0)$ on $\mathcal{M}$. 
\item When the index set $\mathcal{I}$ and the index $i\in\mathcal{I}$ are clear, we denote $\{a_i\}$ as $\{a_i\}_{i\in\mathcal{I}}$.
\item We define the following summation signs for short:
\begin{equation*}
\sum_{\alpha}=\sum_{1\le\alpha\le N},\quad\sum_{\beta\colon\beta\ne\alpha}=\sum_{1\le\beta\le N,\beta\ne\alpha},\quad\sum_{\alpha\ne\beta}=\sum_{1\le\alpha\le N}\sum_{\beta\colon\beta\ne\alpha}.
\end{equation*}
\end{itemize}

Throughout this paper, we adopt the notion of weak solutions for \eqref{eq-VP-point} introduced in \cite{CM10}, in which global well-posedness has been established for each $\varepsilon>0$.
\begin{Definition}[Weak solutions to \eqref{eq-VP-point}]
For initial data satisfying
\begin{align*}
&f_{\varepsilon}^{0}\in L^1\cap L^{\infty}(\mathbb{R}^4),\quad f_{\varepsilon}^{0}\ge 0,\quad\operatorname{supp}f_{\varepsilon}^{0}\subset\{(x,v)\in B(0,R_{\varepsilon}):\min_{\alpha}\vert x-\xi_{\alpha,\varepsilon}^0\vert\ge \sigma_{\varepsilon}\},
\end{align*}
$(f_{\varepsilon},\{\xi_{\alpha,\varepsilon}\})$  is said to be a global weak solution of \eqref{eq-VP-point} if: $f_{\varepsilon},\{\xi_{\alpha,\varepsilon}\}$ satisfy the differential equations in \eqref{eq-VP-point} in the distributional sense and 
\begin{align*}
&f_{\varepsilon}\in L^{\infty}(\mathbb{R}_+,L^1\cap L^{\infty}(\mathbb{R}^4)),\quad f_{\varepsilon}\ge 0,\\
&\operatorname{supp}f_{\varepsilon}(t)\subset\{(x,v)\in B(0,R_{\varepsilon}(t)):\min_{\alpha}\vert x-\xi_{\alpha,\varepsilon}(t)\vert\ge \sigma_{\varepsilon}(t)\},
\end{align*}
for some $R_{\varepsilon}(t)>0$, $\sigma_{\varepsilon}(t)>0$,  which are continuous in time and depend only on the initial data.
\end{Definition}

\begin{Definition}
The energy associated to \eqref{eq-VP-point} is defined as
\begin{align*}
\mathcal{E}(f,\{\xi_{\alpha},\eta_{\alpha}\})=&\frac{1}{2}\iint_{\mathbb{R}^2\times\mathbb{R}^2}\vert v\vert^2f\,\mathrm{d}x\,\mathrm{d}v+\frac{1}{2}\sum_{\alpha}\vert\eta_{\alpha}\vert^2-\frac{1}{2}\iint_{\mathbb{R}^2\times\mathbb{R}^2}\ln\vert x-y\vert\rho(x)\rho(y)\,\mathrm{d}x\,\mathrm{d}y\\
&-\sum_{\alpha}\int_{\mathbb{R}^2}\ln\vert x-\xi_{\alpha}\vert\rho(x)\,\mathrm{d}x-\frac{1}{2}\sum_{\alpha\ne\beta}\ln\vert \xi_{\alpha}-\xi_{\beta}\vert.
\end{align*}
The momentum is defined as
\begin{equation*}
\mathcal{I}(f,\{\xi_{\alpha},\eta_{\alpha}\})=\int_{\mathbb{R}^2}(\vert x+\varepsilon v^{\perp}\vert^2-\varepsilon^2\vert v\vert^2)f\,\mathrm{d}x\,\mathrm{d}v+\sum_{\alpha}(\vert \xi_{\alpha}+\varepsilon\eta_{\alpha}^{\perp}\vert^2-\varepsilon^2\vert\eta_{\alpha}\vert^2).
\end{equation*}
The moment of inertia is defined as
\begin{equation*}
\mathcal{I}_{\rm ine}(f_{\varepsilon},\{\xi_{\alpha,\varepsilon},\eta_{\alpha,\varepsilon}\})=\int_{\mathbb{R}^2}\vert x\vert^2f_{\varepsilon}\,\mathrm{d}x\,\mathrm{d}v+\sum_{\alpha}\vert \xi_{\alpha,\varepsilon}\vert^2.
\end{equation*}
\end{Definition}

Since the solutions to \eqref{eq-mv-E-defect} are measure-valued, we adopt the formulation introduced in \cite{Del91,Sch95,Pou02,Mio19} to clarify the nonlinear term $E_{\rho+\bar{\delta}}^{\perp}\cdot\nabla(\rho+\bar{\delta})=\nabla\cdot \big(E_{\rho+\bar{\delta}}^{\perp}(\rho+\bar{\delta})\big)$.
\begin{Definition}\label{def-symmetric-quadratic-form}
Let $\rho,\mu\in \mathcal{M}_{+}(\mathbb{R}^{2})$. For all $\Phi\in C_{c}^{\infty}(\mathbb{R}^{2})$, the symmetric quadratic form $\mathcal{H}_{\Phi}$ is defined  by
\begin{equation*}
\mathcal{H}_{\Phi}[\rho,\mu]=\frac{1}{2}\iint_{\mathbb{R}^4}H_{\Phi}(x,y)\,\rho(\mathrm{d}x)\,\mu(\mathrm{d}y),
\end{equation*}
where
\begin{equation*}
H_{\Phi}(x,y)=\frac{(x-y)^{\perp}}{\vert x-y\vert^2}\cdot\big(\nabla\Phi(x)-\nabla\Phi(y)\big)\,\text{ if } x\ne y,\quad H_{\Phi}(x,x)=0.
\end{equation*}

The notation $\nabla\cdot(E_{\rho}^{\perp}\rho)$ then is a distribution defined by
\begin{equation}\label{eq-symmetric-form}
\langle\nabla\cdot(E_{\rho}^{\perp}\rho),\Phi\rangle_{\mathcal{D}'(\mathbb{R}^2),\mathcal{D}(\mathbb{R}^2)}:=-\mathcal{H}_{\Phi}[\rho,\rho].
\end{equation}
\end{Definition}

It is obvious that $H_{\Phi}$ is  well defined and bounded on $\mathbb{R}^{2}\times \mathbb{R}^{2}$, vanishes at infinity, and is continuous outside the diagonal $\{(x,x)\vert x\in\mathbb{R}^{2}\}$. If $\rho\in L^{p}(\mathbb{R}^2)$ for $p>2$ such that $E_{\rho}\rho$ is locally integrable, it is obvious that the distribution $\nabla\cdot(E_{\rho}^{\perp}\rho)$ defined above is equivalent to the distributional gradient of $E_{\rho}^{\perp}\rho$.

\begin{Definition}[Weak solutions to \eqref{eq-mv-E-defect}]
$(\rho,\bar{\delta})$ is said to be a weak solution to \eqref{eq-mv-E-defect} with a defect measure $\nu$ if: 
\begin{align*}
&\rho,\bar{\delta}\in L^{\infty}(\mathbb{R}_+,\mathcal{M}_{+}(\mathbb{R}^{2})),\\
&\nu\in L^{\infty}(\mathbb{R}_+,\mathcal{M}(\mathbb{R}^2,\mathbb{R}^{2\times 2})),
\end{align*}
and for all $\Phi\in C_c^{\infty}(\mathbb{R}_+\times\mathbb{R}^2)$
\begin{equation}\label{eq-weak-solution-mv-Euler-defect}
\left\{
\begin{split}
&\int_{\mathbb{R}^2}\Phi(t,x)\,(\rho(t,\mathrm{d}x)+\bar{\delta}(t,\mathrm{d}x))
-\int_{\mathbb{R}^2}\Phi(0,x)\,(\rho(0,\mathrm{d}x)+\bar{\delta}(0,\mathrm{d}x))\\
&=\int_0^t\int_{\mathbb{R}^2}\partial_{t}\Phi(s,x)\,(\rho(s,\mathrm{d}x)+\bar{\delta}(s,\mathrm{d}x))\,\mathrm{d}s+\int_0^t\mathcal{H}_{\Phi(s)}[\rho(s)+\bar{\delta}(s),\rho(s)+\bar{\delta}(s)]\,\mathrm{d}s\\
&\quad+\int_0^t\int_{\mathbb{R}^2}\nabla_x^2\Phi(s,x):\,\nu(s,\mathrm{d}x)\,\mathrm{d}s.
\end{split}\right.
\end{equation}
\end{Definition}

When there is no defect measure, assuming additional regularity on $\rho,\xi_{\alpha}$, the equation \eqref{eq-mv-E-defect} will reduce to the vortex-wave system, which is the extension of \cite[Thm.1.6]{Mio19}. The following proposition can be proved by selecting a special sequence of test functions and taking limits. We put the proof in the appendix.
\begin{Proposition}\label{prn-E0nu-to-E00}
Let $(\rho,\{\xi_{\alpha}\})$ be an accumulation point given by Theorem~\ref{thm-main} and such that $\nu$  vanishes. If moreover $\rho \in L^{\infty}_{\rm{loc}}(\mathbb{R}_{+},L^{p}(\mathbb{R}^{2}))$ for some $p>2$ and  $\{\xi_{\alpha}\}\subset C^{1}(\mathbb{R}_{+},\mathbb{R}^{2})$, then  $(\rho,\{\xi_{\alpha}\})$  satisfies the vortex-wave system \eqref{eq-Vortex-Wave} in the distributional sense.
\end{Proposition}

\subsection{Main result}

Our main result can now be stated as follows.
\begin{Theorem}\label{thm-main}
Let $N>1$ and  $T>0$. Let $f_{\varepsilon}^{0}\ge 0$ and $(f_{\varepsilon}^{0},\{\xi_{\alpha,\varepsilon}^{0},\eta_{\alpha,\varepsilon}^{0}\})$ satisfy the following assumptions: for each $0<\varepsilon<1$, there exist constants $R_{\varepsilon},\sigma_{\varepsilon}>0$ depending only on $\varepsilon$ such that
\begin{equation}\label{eq-initial-assumptions}
\left\{\begin{split}
&\operatorname{supp}f_{\varepsilon}^{0}\subset\{(x,v)\in B(0,R_{\varepsilon}):\min_{\alpha}\vert x-\xi_{\alpha,\varepsilon}^0\vert\ge \sigma_{\varepsilon}\},\\
&K_0:=\sup_{0<\varepsilon<1}\Big(\int_{\mathbb{R}^2}(1+\vert x\vert^2)f_{\varepsilon}^0\,\mathrm{d}x\,\mathrm{d}v+\sum_{\alpha}\vert\xi_{\alpha,\varepsilon}^0\vert\Big)<\infty,\\
&K_1:=\left\vert\sup_{0<\varepsilon<1}\mathcal{E}(f_{\varepsilon}^0,\{\xi_{\alpha,\varepsilon}^0,\eta_{\alpha,\varepsilon}^0\})\right\vert<\infty,\quad\lim_{\varepsilon\to0}\varepsilon^2\| f_{\varepsilon}^0\|_{\infty}=0.
\end{split}\right.
\end{equation}
Let $(f_{\varepsilon},\{\xi_{\alpha,\varepsilon}\})$  denote the corresponding global weak solution of \eqref{eq-VP-point}.

Then there exists a subsequence  $\varepsilon_{n}\rightarrow 0$ as  $n\rightarrow+\infty$ such  that
\begin{itemize}
\item $\rho_{\varepsilon_{n}}\to\rho$ weak-* in $L^{\infty}(\mathbb{R}_{+},\mathcal{M}_+(\mathbb{R}^2))$, $\xi_{\alpha,\varepsilon_n}\to\xi_{\alpha}$ in $C^{r}([0,T],\mathbb{R}^2)$, $\forall r\in[0,1/4)$.

\item $\rho\in L^{\infty}(\mathbb{R}_{+},H^{-1}(\mathbb{R}^2))\cap C^{1/2}([0,T],W^{-2,1}(\mathbb{R}^2))$, $\xi_{\alpha}\in C^{1/2}([0,T],\mathbb{R}^2)$.

\item There exists $M>0$ depending only on $K_0,K_1$ such that $\min_{t\ge 0,\alpha\ne\beta}\vert\xi_{\alpha}(t)-\xi_{\beta}(t)\vert\ge M$.

\item There  exists  $\mu_{0}=\mu_{0}(t,x,\omega)\in L^{\infty}(\mathbb{R}_{+},\mathcal{M}_{+}(\mathbb{R}^2\times\mathbb{S}^1))$  and  $\{\mathbf{M}^{\alpha}\}\subset L^{\infty}(\mathbb{R}_{+},\mathbb{R}^{2\times 2})$ such that $(\rho,\sum_{\alpha}\delta_{\xi_{\alpha}})$ is a weak solution to \eqref{eq-mv-E-defect} with defect measure $\nu$ defined by
\begin{equation}\label{eq-defect-measure}
\nu(t,x):=\int_{\mathbb{S}^1}\omega^{\perp}\otimes \omega \,\mu_{0}(t,x,\mathrm{d}\omega)+\sum_{\alpha} \mathbf{M}^{\alpha}(t)\delta_{\xi_{\alpha}(t)}.
\end{equation}
\end{itemize}
\end{Theorem}

\begin{Remark}
Note that we do not specify the convergence of $f_{\varepsilon}^{0}$, hence the initial condition of the limit equation is defined by the time continuity in distributional sense of the limit points.
\end{Remark}

\begin{Remark}
Very recently, under the non-neutral clusters hypothesis,  Donati and Godard-Cadillac \cite{DG23} prove that $\frac{1}{\zeta+1}$-H\"older regularity for $\zeta$-point-vortex dynamics in the plane is optimal, where $\zeta$ is the index related to the kernel of the Biot-Savart law. In our case $\zeta=1$, it seems that $\frac{1}{2}$-H\"older continuity of the limit points $\xi_{\alpha}$ is sharp. However, if $\rho\equiv 0$, it is proved in the next section that the limit points $\xi_{\alpha}\in C^{0,1}([0,T],\mathbb{R}^2)$, then $\xi_{\alpha}\in C^{\infty}([0,T],\mathbb{R}^2)$ automatically, due to the repulsive interaction between the point charges. When $\rho\ne 0$, it is possible that the plasma particles collide with the point charges as $\varepsilon\to 0$. Hence it might be interesting to study whether $\frac{1}{2}$-H\"older regularity of the limit points $\xi_{\alpha}$ is optimal.
\end{Remark}

\begin{Remark}
Another question is the extension of Saint-Raymond's result \cite{Sai02} to the plasma-charge model. However, as explained in \cite{Mio19}, the absence of the bound for $\iint_{\vert v\vert\le C\varepsilon^{-\gamma}}f_{\varepsilon}\vert v\vert^2\vert x-\xi_{\alpha,\varepsilon}\vert^{-1}\,\mathrm{d}x\,\mathrm{d}v$ with certain $\gamma>0$ makes Saint-Raymond's method not applicable in this paper.
\end{Remark}

\section{The point charge model}\label{sec-2}

Before proving the main result, we consider a simple case that $f_{\varepsilon}\equiv 0$ in \eqref{eq-VP-point}. In this case the model reduces to the point charge model written as follows, which has a concise analysis of the asymptotic behavior as $\varepsilon\to 0$:
\begin{equation}\label{eq-pointCharge}
\left\{
\begin{split}
&\dot{\xi}_{\alpha,\varepsilon}(t)=\frac{\eta_{\alpha,\varepsilon}(t)}{\varepsilon},\quad\dot{\eta}_{\alpha,\varepsilon}(t)=\frac{\eta_{\alpha,\varepsilon}^{\perp}(t)}{\varepsilon^2}+\frac{1}{\varepsilon}F_{\varepsilon}(t,\xi_{\alpha,\varepsilon}(t)),\\
&(\xi_{\alpha,\varepsilon},\eta_{\alpha,\varepsilon})\vert_{t=0}=(\xi_{\alpha,\varepsilon}^0,\eta_{\alpha,\varepsilon}^0),\quad\alpha=1,2,\dots,N.
\end{split}\right.
\end{equation}
The point charge model is a repulsive $N$-body system, which has its own research interest. 

The analysis of \eqref{eq-pointCharge} relies heavily on the quantities $\mathcal{E}_{\varepsilon},\mathcal{I}_{\rm ine}^{\varepsilon}$, called energy and moment of inertia respectively of the system, which are defined as
\begin{align*}
\mathcal{E}_{\varepsilon}(t)&=\frac{1}{2}\sum_{\alpha}\vert\eta_{\alpha,\varepsilon}(t)\vert^2-\frac{1}{2}\sum_{\alpha\ne\beta}\ln\vert \xi_{\alpha,\varepsilon}(t)-\xi_{\beta,\varepsilon}(t)\vert,\\
\mathcal{I}_{\rm ine}^{\varepsilon}(t)&=\sum_{\alpha}\vert \xi_{\alpha,\varepsilon}(t)\vert^2.
\end{align*}

We assume that the initial energy and the initial moment of inertia are uniformly bounded by a constant $K_2>0$:
\begin{equation}\label{eq-energy-moment-of-inertia-assume-point-charge-model}
\vert\sup_{0<\varepsilon<1}\mathcal{E}_{\varepsilon}(0)\vert\le K_2,\quad \sup_{0<\varepsilon<1}\mathcal{I}_{\rm ine}^{\varepsilon}(0)\le K_2.
\end{equation}

We sketch the properties of these quantities as follows. Note the energy is conservative, which means that
\begin{align*}
\vert\mathcal{E}_{\varepsilon}(t)\vert\equiv\vert\mathcal{E}_{\varepsilon}(0)\vert<K_2.
\end{align*}
The moment of inertia satisfies
\begin{align*}
\mathcal{I}_{\rm ine}^{\varepsilon}(t)&\le 4(K_2+N^4)e^{2t},
\end{align*}
by the calculation below and the Gr\"onwall's inequality
\begin{align*}
\left\vert\frac{\mathrm{d}}{\mathrm{d}t}\mathcal{I}_{\rm ine}^{\varepsilon}(t)\right\vert&=\vert\sum_{\alpha}2\xi_{\alpha,\varepsilon}(t)\cdot\eta_{\alpha,\varepsilon}(t)\vert\\
&\le\sum_{\alpha}\vert \xi_{\alpha,\varepsilon}(t)\vert^2+\sum_{\alpha}\vert \eta_{\alpha,\varepsilon}(t)\vert^2\\
&\le\mathcal{I}_{\rm ine}^{\varepsilon}(t)+2\mathcal{E}_{\varepsilon}(0)+\sum_{\alpha\ne\beta}\ln_+\vert \xi_{\alpha,\varepsilon}(t)-\xi_{\beta,\varepsilon}(t)\vert\\
&\le2\mathcal{I}_{\rm ine}^{\varepsilon}(t)+2K_2+N^4.
\end{align*}
The above properties can be used to obtain the upper bound of the velocities and the lower bound of the distance between the point charges. Indeed, combining with the inequality $\ln_+r\le r$ for all $r>0$, we have
\begin{align*}
&\frac{1}{2}\sum_{\alpha}\vert\eta_{\alpha,\varepsilon}(t)\vert^2+\frac{1}{2}\sum_{\alpha\ne\beta}\ln_-\vert \xi_{\alpha,\varepsilon}(t)-\xi_{\beta,\varepsilon}(t)\vert\\
&\le\mathcal{E}_{\varepsilon}(t)+\frac{1}{2}\sum_{\alpha\ne\beta}\ln_+\vert \xi_{\alpha,\varepsilon}(t)-\xi_{\beta,\varepsilon}(t)\vert,\\
&\le\mathcal{E}_{\varepsilon}(0)+N\sum_{\alpha}\vert\xi_{\alpha,\varepsilon}(t)\vert,\\
&\le K_2+2N^2e^{t}\sqrt{K_2+N^4},
\end{align*}
which implies
\begin{align*}
&\sum_{\alpha}\vert\eta_{\alpha,\varepsilon}(t)\vert^2\le 2K_2+4N^2e^{t}\sqrt{K_2+N^4},\\
&\min_{\alpha\ne\beta}\vert \xi_{\alpha,\varepsilon}(t)-\xi_{\beta,\varepsilon}(t)\vert\ge\exp\left(-2K_2-4N^2e^{t}\sqrt{K_2+N^4}\right).
\end{align*}
Hence for arbitrary $T>0$, $\forall t\in[0,T]$, we have constants $M_1,M_2>0$ depending only on $K_2,N,T$ such that
\begin{equation}\label{eq-distance-estimate-point-charge-model}
\sum_{\alpha}\vert\eta_{\alpha,\varepsilon}(t)\vert^2\le M_1,\quad\min_{\alpha\ne\beta}\vert \xi_{\alpha,\varepsilon}(t)-\xi_{\beta,\varepsilon}(t)\vert\ge M_2.
\end{equation}

For each $\varepsilon>0$ and $\xi_{\alpha,\varepsilon}^0\ne\xi_{\beta,\varepsilon}^0$ if $\alpha\ne\beta$, the global existence and uniqueness of classical solutions to the model \eqref{eq-pointCharge} is directly obtained by Cauchy-Lipschitz theorem and the bounds \eqref{eq-distance-estimate-point-charge-model}. Moreover, the $C^1$ solutions of the model \eqref{eq-pointCharge} are automatically $C^{\infty}$. Now we give the analogous result of Theorem~\ref{thm-main} for the model \eqref{eq-pointCharge}
\begin{Theorem}\label{thm-sec-2-main}
Let $t\in[0,T]$. Assume that $\{\xi_{\alpha,\varepsilon}^0,\eta_{\alpha,\varepsilon}^0\}$ satisfies \eqref{eq-energy-moment-of-inertia-assume-point-charge-model}. Assume moreover, there exists a set $\{\xi_{\alpha}^0\}$ such that $\xi_{\alpha}^0\ne\xi_{\beta}^0$ if $\alpha\ne\beta$ and
\begin{equation}
\lim_{\varepsilon\to 0}\sum_{\alpha}\vert\xi_{\alpha,\varepsilon}^0-\xi_{\alpha}^0\vert=0.
\end{equation}
Then the classical solutions $\{\xi_{\alpha,\varepsilon}\}$ of the model \eqref{eq-pointCharge} converge to $\{\xi_{\alpha}\}$ in $C([0,T],\mathbb{R}^2)$ as $\varepsilon\to 0$, where $\{\xi_{\alpha}\}$ is the $C^{0,1}([0,T],\mathbb{R}^2)$ and automatically $C^{\infty}([0,T],\mathbb{R}^2)$ solution of the following point vortex model
\begin{equation}
\dot{\xi_{\alpha}}(t)=F^{\perp}(t,\xi_{\alpha}(t)),\quad \xi_{\alpha}(0)=\xi_{\alpha}^0,\quad\alpha=1,2,\dots,N.
\end{equation}
Moreover, we have the following quantitative estimate:
\begin{equation}\label{eq-sec-2-quantitative-estimate}
\sum_{\alpha}\vert\xi_{\alpha,\varepsilon}(t)-\xi_{\alpha}(t)\vert\le C\left[\sum_{\alpha}\vert\xi_{\alpha,\varepsilon}^0-\xi_{\alpha}^0\vert+\varepsilon\right],
\end{equation}
where $C$ is a constant depending only on $K_2,N,T$.
\end{Theorem}

\begin{proof}
Note
\begin{equation}\label{eq-sec-2-equation-1}
\xi_{\alpha,\varepsilon}(t)+\varepsilon\eta_{\alpha,\varepsilon}^{\perp}(t)=\xi_{\alpha,\varepsilon}^0+\varepsilon(\eta_{\alpha,\varepsilon}^0)^{\perp}+\int_0^tF_{\varepsilon}^{\perp}(s,\xi_{\alpha,\varepsilon}(s))\,\mathrm{d}s.
\end{equation}
For any $\varepsilon_1,\varepsilon_2>0$, we have
\begin{align}
\vert\xi_{\alpha,\varepsilon_1}(t)-\xi_{\alpha,\varepsilon_2}(t)\vert&
\le\vert\xi_{\alpha,\varepsilon_1}^0-\xi_{\alpha,\varepsilon_2}^0\vert+2M_2(\varepsilon_1+\varepsilon_2)\nonumber\\
&\quad+\int_0^t\vert F_{\varepsilon_1}(s,\xi_{\alpha,\varepsilon_1}(s))-F_{\varepsilon_2}(s,\xi_{\alpha,\varepsilon_2}(s))\vert\,\mathrm{d}s\label{eq-sec-2-2.6}
\end{align}
Due to the lower bound in \eqref{eq-distance-estimate-point-charge-model}, there exists a constant $M_3>0$ depending only on $K_2,N,T$ such that
\begin{align*}
\vert F_{\varepsilon_1}(s,\xi_{\alpha,\varepsilon_1}(s))-F_{\varepsilon_2}(s,\xi_{\alpha,\varepsilon_2}(s))\vert\le M_3\sum_{\alpha}\vert\xi_{\alpha,\varepsilon_1}(s)-\xi_{\alpha,\varepsilon_2}(s)\vert.
\end{align*}
Combining with \eqref{eq-sec-2-2.6}, we have
\begin{align*}
\sum_{\alpha}\vert\xi_{\alpha,\varepsilon_1}(t)-\xi_{\alpha,\varepsilon_2}(t)\vert&
\le\sum_{\alpha}\vert\xi_{\alpha,\varepsilon_1}^0-\xi_{\alpha,\varepsilon_2}^0\vert+2NM_2(\varepsilon_1+\varepsilon_2)\\
&\quad+NM_3\int_0^t\sum_{\alpha}\vert\xi_{\alpha,\varepsilon_1}(s)-\xi_{\alpha,\varepsilon_2}(s)\vert\,\mathrm{d}s.
\end{align*}
By the Gr\"onwall's inequality, we have
\begin{align}
\sum_{\alpha}\vert\xi_{\alpha,\varepsilon_1}(t)-\xi_{\alpha,\varepsilon_2}(t)\vert&
\le\left[\sum_{\alpha}\vert\xi_{\alpha,\varepsilon_1}^0-\xi_{\alpha,\varepsilon_2}^0\vert+2NM_2(\varepsilon_1+\varepsilon_2)\right](1+NM_3te^{NM_3t})\nonumber\\
&\le C\left[\sum_{\alpha}\vert\xi_{\alpha,\varepsilon_1}^0-\xi_{\alpha,\varepsilon_2}^0\vert+\varepsilon_1+\varepsilon_2\right],\label{eq-sec-2-quantitative-estimate-2}
\end{align}
where $C$ is a constant depending only on $K_2,N,T$. Hence $\xi_{\alpha,\varepsilon}(\cdot)$ is a Cauchy sequence in $C([0,T],\mathbb{R}^2)$ as $\varepsilon\to 0$. We denote its limit point as $\xi_{\alpha}(\cdot)$. The estimate \eqref{eq-sec-2-quantitative-estimate} can be deduced by \eqref{eq-sec-2-quantitative-estimate-2}. Moreover, there holds
\begin{equation*}
\min_{\alpha\ne\beta}\vert \xi_{\alpha}(t)-\xi_{\beta}(t)\vert\ge M_2.
\end{equation*}
We can deduce that the equation \eqref{eq-sec-2-equation-1} converges to
\begin{equation*}
\xi_{\alpha}(t)=\xi_{\alpha}^0+\int_0^tF^{\perp}(s,\xi_{\alpha}(s))\,\mathrm{d}s,
\end{equation*}
which implies $\xi_{\alpha}$ is a $C^{0,1}$ solution of the point vortex model. Hence $\xi_{\alpha}$ is a.e.~differentiable and
\begin{equation*}
\dot{\xi_{\alpha}}(t)=F^{\perp}(t,\xi_{\alpha}(t)).
\end{equation*}
However, the R.H.S.~in the above equality belongs to $C([0,T],\mathbb{R}^2)$. Hence $\xi_{\alpha}$ is actually  $C^{1}([0,T],\mathbb{R}^2)$. Then it implies that the R.H.S.~in the above equality belongs to $C^1([0,T],\mathbb{R}^2)$. This once again shows that $\xi_{\alpha}\in C^{2}([0,T],\mathbb{R}^2)$. By induction, we finally obtain that $\xi_{\alpha}\in C^{\infty}([0,T],\mathbb{R}^2)$. 
\end{proof}

\begin{Remark}\label{rem-sec-2-rem-2-1}
The existence of $\{\xi_{\alpha}\}$ in Theorem~\ref{thm-sec-2-main} can be proved also by the compactness argument. Indeed, assume $0\le t_1<t_2\le T$, by \eqref{eq-sec-2-equation-1}, we have
\begin{equation*}
\vert\xi_{\alpha,\varepsilon}(t_2)-\xi_{\alpha,\varepsilon}(t_1)\vert\le\varepsilon\vert\eta_{\alpha,\varepsilon}(t_1)-\eta_{\alpha,\varepsilon}(t_2)\vert+\int_{t_1}^{t_2}\vert F_{\varepsilon}(s,\xi_{\alpha,\varepsilon}(s))\vert\,\mathrm{d}s.
\end{equation*}
By the uniform bounds \eqref{eq-distance-estimate-point-charge-model} and the equations in \eqref{eq-pointCharge}, we have
\begin{equation*}
\vert\eta_{\alpha,\varepsilon}(t_1)-\eta_{\alpha,\varepsilon}(t_2)\vert\le C\varepsilon^{-2}(t_2-t_1).
\end{equation*}
By considering the cases $\varepsilon>\sqrt{t_2-t_1}$ and $\varepsilon\le\sqrt{t_2-t_1}$, we have
\begin{equation*}
\varepsilon\vert\eta_{\alpha,\varepsilon}(t_1)-\eta_{\alpha,\varepsilon}(t_2)\vert\le C\sqrt{t_2-t_1}.
\end{equation*}
Hence
\begin{equation*}
\vert\xi_{\alpha,\varepsilon}(t_2)-\xi_{\alpha,\varepsilon}(t_1)\vert\le C\sqrt{t_2-t_1},
\end{equation*}
where $C$ is a constant depending only on $K_2,N,T$. We obtain that $\xi_{\alpha,\varepsilon}(\cdot)$ is uniformly bounded in $C^{1/2}([0,T],\mathbb{R}^2)$. The Arzel\`a-Ascoli theorem is available.
\end{Remark}

\begin{Remark}\label{rem-sec-2-rem-2-2}
When $f_{\varepsilon}\ne 0$, the R.H.S.~of \eqref{eq-sec-2-equation-1} has an extra term $\int_0^tE_{\varepsilon}^{\perp}(s,\xi_{\alpha,\varepsilon}(s))\,\mathrm{d}s$, which is the key term to be dealt with in Sec.~\ref{sec-est-point-singular-field}. In this case, the compactness argument in Remark~\ref{rem-sec-2-rem-2-1} is not obvious, since we only obtain the estimate on the extra term like $\int_0^t\vert E_{\varepsilon}^{\perp}(s,\xi_{\alpha,\varepsilon}(s))\vert\,\mathrm{d}s\lesssim\varepsilon^{-1}t^{1/2}$. In \cite{Mio19}, Miot utilizes the benefit of the symmetric quadratic form given in Definition~\ref{def-symmetric-quadratic-form} and a contradiction argument to obtain that
\begin{equation*}
\vert\xi_{\alpha,\varepsilon}(t_2)-\xi_{\alpha,\varepsilon}(t_1)\vert\lesssim \sqrt{t_2-t_1}+\varepsilon^{1/3}.
\end{equation*}
Then a discontinuous version of Arzel\`a-Ascoli theorem is available to obtain the compactness of $\xi_{\alpha,\varepsilon}$ in $C([0,T],\mathbb{R}^2)$.

Roughly speaking, the contradiction argument is based on the construction of a test function $\Phi\in C_c^{\infty}(\mathbb{R}^2)$ such that $\Phi(\xi_{\alpha,\varepsilon}(t_1))=0$,  $\Phi(\xi_{\alpha,\varepsilon}(t_2))=1$ for all $\alpha=1,\dots,N$. This construction can be achieved for the case $N=1$. However, if $N>1$, due to the strong cyclotron effect as $\varepsilon\to 0$, it might be inevitable that $\xi_{\alpha,\varepsilon}(t_2)=\xi_{\beta,\varepsilon}(t_1)$ for some $\alpha\ne\beta$, $t_1,t_2\in[0,T]$, and it leads to $0=\Phi(\xi_{\alpha,\varepsilon}(t_2))=\Phi(\xi_{\beta,\varepsilon}(t_1))=1$, which implies that such test function does not always exist.
\end{Remark}

\section{Preliminary estimates}

\subsection{ First dynamical estimates}

The characteristic associated to \eqref{eq-VP-point} is the solution of the following ODEs:
\begin{equation}
\left\{
\begin{split}
&\frac{\mathrm{d}}{\mathrm{d}s}X_{\varepsilon}(s,t,x,v)=\frac{V_{\varepsilon}(s,t,x,v)}{\varepsilon},\\
&\frac{\mathrm{d}}{\mathrm{d}s}V_{\varepsilon}(s,t,x,v)=\frac{V_{\varepsilon}^{\perp}(s,t,x,v)}{\varepsilon^2}+\frac{(E_{\varepsilon}+F_{\varepsilon})(s,X_{\varepsilon}(s,t,x,v))}{\varepsilon},\\
&(X_{\varepsilon},V_{\varepsilon})(t,t,x,v)=(x,v)\in\operatorname{supp}f_{\varepsilon}(t).
\end{split}\right.
\end{equation}
Note the regularity of $E_{\varepsilon},F_{\varepsilon}$ in this paper for each fixed $\varepsilon>0$ is enough to verify the conditions in \cite{DL89ODE}. Hence the solution to the system \eqref{eq-VP-point} can be represented by
\begin{equation*}
f_{\varepsilon}(t,x,v)=f_{\varepsilon}^0(X_{\varepsilon}(0,t,x,v),V_{\varepsilon}(0,t,x,v)).
\end{equation*}
By this representation and the measure-preserving property of $(x,v)\mapsto (X_{\varepsilon},V_{\varepsilon})(s,t,x,v)$, we have the $L^p$ norms of $f_{\varepsilon}$ are conserved:
\begin{equation}\label{eq-Lp-conse}
\|f_{\varepsilon}(t)\|_{p}=\|f_{\varepsilon}^0\|_{p},\,~~\forall~1\le p\le \infty.
\end{equation}

Moreover, the energy and the momentum associated to the system \eqref{eq-VP-point} are conserved:
\begin{equation}\label{eq-energy-moment-conse}
\begin{split}
\mathcal{E}(f_{\varepsilon},\{\xi_{\alpha,\varepsilon},\eta_{\alpha,\varepsilon}\})=\mathcal{E}(f_{\varepsilon}^0,\{\xi_{\alpha,\varepsilon}^0,\eta_{\alpha,\varepsilon}^0\}),\\
\mathcal{I}(f_{\varepsilon},\{\xi_{\alpha,\varepsilon},\eta_{\alpha,\varepsilon}\})=\mathcal{I}(f_{\varepsilon}^0,\{\xi_{\alpha,\varepsilon}^0,\eta_{\alpha,\varepsilon}^0\}).
\end{split}
\end{equation}
Since $f_{\varepsilon}$ has compact velocity support, these conservative laws can be deduced rigorously by the weak formula of the solution to \eqref{eq-VP-point} and the dominated convergence theorem, see \cite{Mio16,Mio19} or \cite{WZ23VPenergy}. We omit the details for the sake of simplicity.

Thanks to the conservation laws \eqref{eq-Lp-conse} and \eqref{eq-energy-moment-conse}, we have the following a priori estimates.
\begin{Proposition}\label{prn-priori}
There exists a constant $C$ depending only on $K_0,K_1$ such that
\begin{align}
&\sup_{t\ge 0,\varepsilon>0}\iint_{\mathbb{R}^2\times\mathbb{R}^2}(\vert x\vert^2+\vert v\vert^2)f_{\varepsilon}(t,x,v)\,\mathrm{d}x\,\mathrm{d}v+\sum_{\alpha}(\vert\xi_{\alpha,\varepsilon}(t)\vert+\vert\eta_{\alpha,\varepsilon}(t)\vert)\le C,\label{eq-priori-1}\\
&\sup_{t\ge 0}\|\rho_{\varepsilon}(t)\|_{2}\le C\| f_{\varepsilon}^0\|_{\infty}^{1/2},\label{eq-priori-2}\\
&\sup_{t\ge 0,\varepsilon>0}\sum_{\alpha}\int_{\mathbb{R}^2}\big\vert\ln\vert x-\xi_{\alpha,\varepsilon}(t)\vert\big\vert\rho_{\varepsilon}(t,x)\,\mathrm{d}x\le C,\label{eq-priori-3}\\
&\sup_{t\ge 0,\varepsilon>0}\iint_{\mathbb{R}^2\times\mathbb{R}^2}\big\vert\ln\vert x-y\vert\big\vert\rho_{\varepsilon}(t,x)\rho_{\varepsilon}(t,y)\,\mathrm{d}x\,\mathrm{d}y\le C\label{eq-priori-4}.
\end{align}
And the non-concentration property holds:
\begin{equation}\label{eq-estimate-non-concentration}
\sup\limits_{t\in\mathbb{R}_{+}}\sup\limits_{0<\varepsilon<1}\sup\limits_{x_{0}\in\mathbb{R}^2}\sup\limits_{0<r<1/2}\vert\ln r\vert^{1/2}\int_{B(x_{0},r)}\rho_{\varepsilon}(t,x)\,\mathrm{d}x\le C.
\end{equation}
Moreover, there exists a constant $M>0$ depending only on $K_0,K_1$ such that
\begin{equation}\label{eq-priori-positive-distance}
\min_{t\ge 0,\varepsilon>0,\alpha\ne\beta}\vert\xi_{\alpha,\varepsilon}(t)-\xi_{\beta,\varepsilon}(t)\vert\ge M.
\end{equation}
\end{Proposition}
\begin{proof}

\eqref{eq-priori-1}-\eqref{eq-estimate-non-concentration} have been proved in \cite[Cor.2.2]{Mio19} for the case of single point charge, which can be extended to the case of multi-point charges as follows.  

Note the positive part of $-\ln$ is $\ln_-$. Accordingly, we define the sum of the positive components of energy as
\begin{align*}
\mathcal{E}_{\rm pos}(f_{\varepsilon},\{\xi_{\alpha,\varepsilon},\eta_{\alpha,\varepsilon}\}):=&\,\frac{1}{2}\iint_{\mathbb{R}^2\times\mathbb{R}^2}\vert v\vert^2f_{\varepsilon}\,\mathrm{d}x\,\mathrm{d}v+\frac{1}{2}\sum_{\alpha}\vert\eta_{\alpha,\varepsilon}\vert^2\\
&+\frac{1}{2}\iint_{\mathbb{R}^2\times\mathbb{R}^2}\ln_{-}\vert x-y\vert\rho_{\varepsilon}(t,x)\rho_{\varepsilon}(t,y)\,\mathrm{d}x\,\mathrm{d}y\\
&+\sum_{\alpha}\int_{\mathbb{R}^2}\ln_{-}\vert x-\xi_{\alpha,\varepsilon}(t)\vert\rho_{\varepsilon}(t,x)\,\mathrm{d}x+\frac{1}{2}\sum_{\alpha\ne\beta}\ln_{-}\vert \xi_{\alpha}-\xi_{\beta}\vert.
\end{align*}
By the energy conservation and the inequality $\ln_+r\le r$ for all $r>0$, we have
\begin{align}
\mathcal{E}_{\rm pos}(f_{\varepsilon},\{\xi_{\alpha,\varepsilon},\eta_{\alpha,\varepsilon}\})&\le\mathcal{E}(f_{\varepsilon},\{\xi_{\alpha,\varepsilon},\eta_{\alpha,\varepsilon}\})+\frac{1}{2}\iint_{\mathbb{R}^2\times\mathbb{R}^2}\ln_{+}\vert x-y\vert\rho_{\varepsilon}(x)\rho_{\varepsilon}(y)\,\mathrm{d}x\,\mathrm{d}y\nonumber\\
&\quad+\sum_{\alpha}\int_{\mathbb{R}^2}\ln_{+}\vert x-\xi_{\alpha,\varepsilon}\vert\rho_{\varepsilon}(x)\,\mathrm{d}x+\frac{1}{2}\sum_{\alpha\ne\beta}\ln_{+}\vert \xi_{\alpha,\varepsilon}-\xi_{\beta,\varepsilon}\vert\nonumber\\
&\le\mathcal{E}(f_{\varepsilon}^0,\{\xi_{\alpha,\varepsilon}^0,\eta_{\alpha,\varepsilon}^0\})+\frac{1}{2}\iint_{\mathbb{R}^2\times\mathbb{R}^2}\left(\vert x\vert+\vert y\vert\right)\rho_{\varepsilon}(x)\rho_{\varepsilon}(y)\,\mathrm{d}x\,\mathrm{d}y\nonumber\\
&\quad+\sum_{\alpha}\int_{\mathbb{R}^2}\left(\vert x\vert+\vert\xi_{\alpha,\varepsilon}\vert\right)\rho_{\varepsilon}(x)\,\mathrm{d}x+\frac{1}{2}\sum_{\alpha\ne\beta}\left(\vert \xi_{\alpha,\varepsilon}\vert+\vert\xi_{\beta,\varepsilon}\vert\right)\nonumber\\
&\lesssim 1+[\mathcal{I}_{\rm ine}(f_{\varepsilon},\{\xi_{\alpha,\varepsilon},\eta_{\alpha,\varepsilon}\})]^{\frac{1}{2}}\label{eq-2024-2-1-01}.
\end{align}
Here we use the notation $a\lesssim b$ to denote $a\le C b$, where the constant $C$ depends only on $\mathcal{E}(f_{\varepsilon}^0,\{\xi_{\alpha,\varepsilon}^0,\eta_{\alpha,\varepsilon}^0\})$ and $\|f_{\varepsilon}^0\|_{L^1}$.

Notice
\begin{equation*}
\mathcal{I}_{\rm ine}(f_{\varepsilon},\{\xi_{\alpha,\varepsilon},\eta_{\alpha,\varepsilon}\})=\mathcal{I}(f_{\varepsilon},\{\xi_{\alpha,\varepsilon},\eta_{\alpha,\varepsilon}\})-2\varepsilon \int_{\mathbb{R}^2}x\cdot v^{\perp}f_{\varepsilon}\,\mathrm{d}x\,\mathrm{d}v-2\varepsilon\sum_{\alpha}\xi_{\alpha,\varepsilon}\cdot\eta_{\alpha,\varepsilon}^{\perp}.
\end{equation*}
By the conservation of the momentum, the Young's inequality for products, the definitions of $\mathcal{I}_{\rm ine},\mathcal{E}_{\rm pos}$ and \eqref{eq-2024-2-1-01}, we have
\begin{align*}
\mathcal{I}_{\rm ine}(f_{\varepsilon},\{\xi_{\alpha,\varepsilon},\eta_{\alpha,\varepsilon}\})&\le\mathcal{I}(f_{\varepsilon}^0,\{\xi_{\alpha,\varepsilon}^0,\eta_{\alpha,\varepsilon}^0\})+\frac{1}{2}\int_{\mathbb{R}^2}\vert x\vert^2f_{\varepsilon}\,\mathrm{d}x\,\mathrm{d}v\\
&\quad+2\varepsilon^2\int_{\mathbb{R}^2}\vert v\vert^2f_{\varepsilon}\,\mathrm{d}x\,\mathrm{d}v+\frac{1}{2}\sum_{\alpha}\vert\xi_{\alpha,\varepsilon}\vert^2+2\varepsilon^2\sum_{\alpha}\vert\eta_{\alpha,\varepsilon}\vert^2\\
&\le\mathcal{I}(f_{\varepsilon}^0,\{\xi_{\alpha,\varepsilon}^0,\eta_{\alpha,\varepsilon}^0\})+\frac{1}{2}\mathcal{I}_{\rm ine}(f_{\varepsilon},\{\xi_{\alpha,\varepsilon},\eta_{\alpha,\varepsilon}\})+4\varepsilon^2\mathcal{E}_{\rm pos}(f_{\varepsilon},\{\xi_{\alpha,\varepsilon},\eta_{\alpha,\varepsilon}\})\\
&\le C+\frac{1}{2}\mathcal{I}_{\rm ine}(f_{\varepsilon},\{\xi_{\alpha,\varepsilon},\eta_{\alpha,\varepsilon}\})+C[\mathcal{I}_{\rm ine}(f_{\varepsilon},\{\xi_{\alpha,\varepsilon},\eta_{\alpha,\varepsilon}\})]^{\frac{1}{2}},
\end{align*}
which implies $\mathcal{I}_{\rm ine}(f_{\varepsilon},\{\xi_{\alpha,\varepsilon},\eta_{\alpha,\varepsilon}\})\le C$. By \eqref{eq-2024-2-1-01} again we have $\mathcal{E}_{\rm pos}(f_{\varepsilon},\{\xi_{\alpha,\varepsilon},\eta_{\alpha,\varepsilon}\})\le C$ and \eqref{eq-priori-1} follows immediately. To prove \eqref{eq-priori-3}-\eqref{eq-priori-4}, we have
\begin{align*}
&\sum_{\alpha}\int_{\mathbb{R}^2}\big\vert\ln\vert x-\xi_{\alpha,\varepsilon}(t)\vert\big\vert\rho_{\varepsilon}(t,x)\,\mathrm{d}x+\iint_{\mathbb{R}^2\times\mathbb{R}^2}\big\vert\ln\vert x-y\vert\big\vert\rho_{\varepsilon}(t,x)\rho_{\varepsilon}(t,y)\,\mathrm{d}x\,\mathrm{d}y\\
&\le 2\mathcal{E}_{\rm pos}(f_{\varepsilon},\{\xi_{\alpha,\varepsilon},\eta_{\alpha,\varepsilon}\})+\sum_{\alpha}\int_{\mathbb{R}^2}\ln_{+}\vert x-\xi_{\alpha,\varepsilon}(t)\vert\rho_{\varepsilon}(t,x)\,\mathrm{d}x\\
&\quad+\iint_{\mathbb{R}^2\times\mathbb{R}^2}\ln_{+}\vert x-y\vert\rho_{\varepsilon}(t,x)\rho_{\varepsilon}(t,y)\,\mathrm{d}x\,\mathrm{d}y\\
&\le 2\mathcal{E}_{\rm pos}(f_{\varepsilon},\{\xi_{\alpha,\varepsilon},\eta_{\alpha,\varepsilon}\})+\sum_{\alpha}\int_{\mathbb{R}^2}\left(\vert x\vert^2+\vert\xi_{\alpha,\varepsilon}(t)\vert^2+2\right)\rho_{\varepsilon}(t,x)\,\mathrm{d}x\\
&\quad+\iint_{\mathbb{R}^2\times\mathbb{R}^2}\left(\vert x\vert^2+\vert y\vert^2+2\right)\rho_{\varepsilon}(t,x)\rho_{\varepsilon}(t,y)\,\mathrm{d}x\,\mathrm{d}y\\
&\le C.
\end{align*}
To prove \eqref{eq-priori-positive-distance}, we have
\begin{align*}
\sum_{\alpha\ne\beta}\ln_{-}\vert \xi_{\alpha,\varepsilon}(t)-\xi_{\beta,\varepsilon}(t)\vert\le 2\mathcal{E}_{\rm pos}(f_{\varepsilon},\{\xi_{\alpha,\varepsilon},\eta_{\alpha,\varepsilon}\})\le C.
\end{align*}
We fix the constant $C$ in the last inequality to be $\tilde{M}>0$ depending only on $K_0,K_1$, hence  we obtain
$\vert\xi_{\alpha,\varepsilon}(t)-\xi_{\beta,\varepsilon}(t)\vert\ge e^{-\tilde{M}}=:M$.

Finally, \eqref{eq-priori-2} is classical by the interpolation inequality, and \eqref{eq-estimate-non-concentration} is the consequence of \eqref{eq-priori-1}, \eqref{eq-priori-4}.

\end{proof}

\subsection{$L_t^{\infty}H_x^{-1}$-bound of the densities}\label{subsec-LH-1-bound-rho}

In this subsection, we adopt the method in \cite{Mio16} to obtain the uniform $L_t^{\infty}H_x^{-1}$-bound of $\rho_{\varepsilon}$. Firstly, we decompose $\rho_{\varepsilon}$ into a zero-mean part $\widetilde{\rho}_{\varepsilon}$ and a smooth part $\bar{\rho}_{\varepsilon}$. Let $\bar{\rho}_{\varepsilon}(x)=\|f_{\varepsilon}^0\|_{1}\psi(x)$, where $\psi\in C_c^{\infty}(\mathbb{R}^2,\mathbb{R}_+)$ with $\operatorname{supp}\psi\subset B(0,1)$ and $\|\psi\|_{1}=1$. Denote 
\begin{equation*}
\widetilde{\rho}_{\varepsilon}=\rho_{\varepsilon}-\bar{\rho}_{\varepsilon},\quad\bar{E}_{\varepsilon}=\frac{x}{\vert x\vert^2}*\bar{\rho}_{\varepsilon},\quad\widetilde{E}_{\varepsilon}=\frac{x}{\vert x\vert^2}*\widetilde{\rho}_{\varepsilon}.
\end{equation*}
Notice $\|\bar{\rho}_{\varepsilon}\|_{L^1\cap L^{\infty}}\le (1+\|\psi\|_{\infty})\|f_{\varepsilon}^0\|_{1}$. Therefore, we only need to deal with the zero-mean parts $\widetilde{\rho}_{\varepsilon}$ next.

We recall the classical result from potential theory, see e.g., \cite[Lem.3]{GNPS05}.
\begin{Lemma}\label{lem-2D-potential}
Let $\rho\in L^2(\mathbb{R}^2)$ be such that
\begin{equation*}
\int_{\mathbb{R}^2}(1+\vert x\vert)\vert\rho(x)\vert\,\mathrm{d}x<\infty,\quad\int_{\mathbb{R}^2}\rho(x)\,\mathrm{d}x=0.
\end{equation*}
Consider the potential $U_{\rho}=-\ln\vert x\vert*\rho$. Then $U_{\rho}\in C_0(\mathbb{R}^2)$ and $\nabla U_{\rho}\in L^2(\mathbb{R}^2)$. In particular, the identity holds:
\begin{equation*}
2\pi\int_{\mathbb{R}^2}\rho(x)U_{\rho}(x)\,\mathrm{d}x=\int_{\mathbb{R}^2}\vert\nabla U_{\rho}(x)\vert^2\,\mathrm{d}x.
\end{equation*}
\end{Lemma}

Follow the technique in \cite{Mio16}, we have the uniform bound of the zero-mean parts.
\begin{Proposition}\label{prn-H-1-bound-density}
$\widetilde{\rho}_{\varepsilon}$ is uniformly bounded in $L^{\infty}(\mathbb{R}_+,H^{-1}(\mathbb{R}^2))$ and
\begin{align*}
\sup_{t\ge0,\varepsilon>0}\int_{\mathbb{R}^2}\vert\widetilde{E}_{\varepsilon}\vert^2\,\mathrm{d}x\le C.
\end{align*}
\end{Proposition}
\begin{proof}
Notice $\widetilde{\rho}_{\varepsilon}$ satisfies the assumption in Lemma~\ref{lem-2D-potential} and
\begin{align*}
\frac{1}{2\pi}\int_{\mathbb{R}^2}\vert\widetilde{E}_{\varepsilon}\vert^2\,\mathrm{d}x&=-\iint_{\mathbb{R}^2\times\mathbb{R}^2}\ln\vert x-y\vert\rho_{\varepsilon}(x)\rho_{\varepsilon}(y)\,\mathrm{d}x\,\mathrm{d}y-\iint_{\mathbb{R}^2\times\mathbb{R}^2}\ln\vert x-y\vert\bar{\rho}_{\varepsilon}(x)\bar{\rho}_{\varepsilon}(y)\,\mathrm{d}x\,\mathrm{d}y\\
&\quad+2\iint_{\mathbb{R}^2\times\mathbb{R}^2}\ln\vert x-y\vert\rho_{\varepsilon}(x)\bar{\rho}_{\varepsilon}(y)\,\mathrm{d}x\,\mathrm{d}y\\
&\le\iint_{\mathbb{R}^2\times\mathbb{R}^2}\big\vert\ln\vert x-y\vert\big\vert\rho_{\varepsilon}(x)\rho_{\varepsilon}(y)+\|f_{\varepsilon}^0\|_{1}^2\|\psi\|_{\infty}^2\iint_{B(0,1)\times B(0,1)}\big\vert\ln\vert x-y\vert\big\vert\\
&\quad+2\iint_{\mathbb{R}^2\times\mathbb{R}^2}\ln_+\vert x-y\vert\rho_{\varepsilon}(x)\bar{\rho}_{\varepsilon}(y).
\end{align*}
Notice $\ln_+\vert x-y\vert\le \vert x\vert+\vert y\vert$ and $(x,y)\mapsto\ln\vert x-y\vert$ is locally integrable in $\mathbb{R}^4$. By \eqref{eq-priori-1} and \eqref{eq-priori-4} we have
\begin{align*}
\sup_{t\ge0,\varepsilon>0}\int_{\mathbb{R}^2}\vert\widetilde{E}_{\varepsilon}\vert^2\,\mathrm{d}x\le C,
\end{align*}
which implies that $\widetilde{\rho}_{\varepsilon}$ is uniformly bounded in $L^{\infty}(\mathbb{R}_+,H^{-1}(\mathbb{R}^2))$. The proof is complete.
\end{proof}

\begin{Remark}
Notice by the H-L-S theorem, $\bar{E}_{\varepsilon}\in L_{\rm w}^2\cap L^{q}\not\subset L^2(\mathbb{R}^2)$ for any $2<q<\infty$, where $L_{\rm w}^2(\mathbb{R}^2)$ denotes the weak Lebesgue space, i.e., the set of all measurable functions $g$ satisfying
\begin{equation*}
\sup_{\lambda>0}\lambda^{2}\mathcal{L}(\{x\in\mathbb{R}^2\,:\,\vert g(x)\vert>\lambda\})<\infty.
\end{equation*}
Hence the proposition above does not imply $E_{\varepsilon}\in L^2$.
\end{Remark}

Combining the estimate of $\bar{\rho}_{\varepsilon}$ and Proposition~\ref{prn-H-1-bound-density}, we obtain the uniform $L_t^{\infty}H_x^{-1}$-bound of $\rho_{\varepsilon}$.

\section{Estimates for the point charges}\label{sec-est-point-singular-field}
In this section we focus on estimates for the point charges. In particular we obtain the bound of time integral of the field $E_{\varepsilon}$ along the trajectories of point charges, which will be used in the next section. 

Firstly, we define the pointwise energy function as
\begin{equation*}
h_{\varepsilon}(t,x,v)=\frac{\vert v\vert^2}{2}+\sum_{\alpha}\big(\vert x-\xi_{\alpha,\varepsilon}(t)\vert-\ln\vert x-\xi_{\alpha,\varepsilon}(t)\vert\big)+K,
\end{equation*}
where $K>1$ is a constant. Notice $h_{\varepsilon}(t,x,v)\ge \vert v\vert^2/2$. We denote
\begin{equation*}
H_{k,\varepsilon}(t)=\sup_{0\le s\le t}\widetilde{H}_{k,\varepsilon}(s),\quad \widetilde{H}_{k,\varepsilon}(t)=\iint_{\mathbb{R}^2\times\mathbb{R}^2} h_{\varepsilon}^{k/2}f_{\varepsilon}(t,x,v)\,\mathrm{d}x\,\mathrm{d}v.
\end{equation*}

We have the following kinetic interpolation inequality, which is a trivial extension of \cite[Lem.3.1]{GS99}.
\begin{Lemma}\label{lem-interpo-Hk}
For all $0\le k\le l$, we have
\begin{equation*}
\Big\|\int_{\mathbb{R}^2} h_{\varepsilon}^{k/2}f_{\varepsilon}(t,\cdot,v)\,\mathrm{d}v\Big\|_{\frac{l+2}{k+2}}\le C \|f_{\varepsilon}^0\|_{\infty}^{\frac{l-k}{l+2}}H_{l,\varepsilon}(t)^{\frac{k+2}{l+2}}.
\end{equation*}
In particular, $H_{k,\varepsilon}(t)\le C'$ for all $0\le k\le 2$, where $C$ depends only on $k,l$; $C'$ depends only on $K_0,K_1$.
\end{Lemma}

For the sake of simplicity, we will use the following shorthand in the sequel of this section:
\begin{align*}
&(X_{\varepsilon}(s),V_{\varepsilon}(s))=(X_{\varepsilon},V_{\varepsilon})(s,0,x,v),\\
&\mathbf{h}_{\varepsilon}(s)=h_{\varepsilon}(s,X_{\varepsilon}(s),V_{\varepsilon}(s)),
\end{align*}
where $(x,v)$ belongs to the support of $f_{\varepsilon}^0$. We adopt the approach established recently in \cite{WZ23} to obtain the following proposition. The proof is in the appendix.

\begin{Proposition}\label{prn-estimate-Lk}
We define $k$-th order singular moment by
\begin{equation*}
L_{k,\varepsilon}(t):=\sum_{\alpha}\int_0^t\iint_{\mathbb{R}^2\times\mathbb{R}^2} \frac{\mathbf{h}_{\varepsilon}(s)^{k/2}f_{\varepsilon}^0}{\vert X_{\varepsilon}(s)-\xi_{\alpha,\varepsilon}(s)\vert}\,\mathrm{d}x\,\mathrm{d}v\,\mathrm{d}s.
\end{equation*}
Then  for all $0\le k<l$ we have
\begin{align*}
L_{k,\varepsilon}(t)&\le C(\varepsilon+\varepsilon^{-1}t) H_{k+1,\varepsilon}(t)+C\|f_{\varepsilon}^0\|_{\infty}^{\frac{l-k}{l+2}}H_{l,\varepsilon}(t)^{\frac{k+2}{l+2}}\int_0^t\|E_{\varepsilon}(s)\|_{\frac{l+2}{l-k}}\,\mathrm{d}s\\
&\qquad+CL_{k-1,\varepsilon}(t)+C\big(L_{0,\varepsilon}(t)+t\big)H_{k,\varepsilon}(t),
\end{align*}
where $C$ depends only on $N,K_0,K_1$ and $l,k$.
\end{Proposition}

Now we have the main result in the section. 
\begin{Corollary}\label{coro-key-L0}
For all $t\ge 0$, there holds
\begin{equation}\label{eq-coro-key-L0-1}
L_{0,\varepsilon}(t)\le C(\varepsilon^{-1}t+\varepsilon),
\end{equation}
where $C$ depends only on $N,K_0,K_1$.

In particular, there exists $t_0>0$ depending only on $N,K_0,K_1$ such that for all $0<t<t_0$, there holds
\begin{equation}\label{eq-coro-key-L0-2}
L_{0,\varepsilon}(t)\le Ct^{1/2}\varepsilon^{-1}.
\end{equation}
\end{Corollary}
\begin{proof}
Let $R,\sigma>0$ to be determined later, we split $\mathbb{R}^2\times\mathbb{R}^2$ into three parts $A_1,A_2,A_3$ as
\begin{align*}
A_1:=&\{(x,v):\vert X_{\varepsilon}(s)-\xi_{\alpha,\varepsilon}(s)\vert >\sigma\},\\
A_2:=&\{(x,v):\mathbf{h}_{\varepsilon}(s)>R\},\\
A_3:=&\{(x,v):\vert X_{\varepsilon}(s)-\xi_{\alpha,\varepsilon}(s)\vert \le\sigma,\mathbf{h}_{\varepsilon}(s)\le R\}.
\end{align*}
Then we have
\begin{align*}
&\int_0^t\iint_{\mathbb{R}^2\times\mathbb{R}^2} \frac{f_{\varepsilon}^0(x,v)}{\vert X_{\varepsilon}(s)-\xi_{\alpha,\varepsilon}(s)\vert}\,\mathrm{d}x\,\mathrm{d}v\,\mathrm{d}s\\
&=\int_0^t\iint_{A_1}+\int_0^t\iint_{A_2}+\int_0^t\iint_{A_3}\frac{f_{\varepsilon}^0(x,v)}{\vert X_{\varepsilon}(s)-\xi_{\alpha,\varepsilon}(s)\vert}\,\mathrm{d}x\,\mathrm{d}v\,\mathrm{d}s\\
&\le \sigma^{-1}\|f_{\varepsilon}^0\|_1t+R^{-1/2}\int_0^t\iint_{\mathbb{R}^2\times\mathbb{R}^2} \frac{\mathbf{h}_{\varepsilon}(s)^{1/2}f_{\varepsilon}^0(x,v)}{\vert X_{\varepsilon}(s)-\xi_{\alpha,\varepsilon}(s)\vert}\,\mathrm{d}x\,\mathrm{d}v\,\mathrm{d}s\\
&\quad+\int_0^t\iint_{\vert x-\xi_{\alpha,\varepsilon}(s)\vert \le\sigma,h_{\varepsilon}(s,x,v)\le R} \frac{f_{\varepsilon}(s,x,v)}{\vert x-\xi_{\alpha,\varepsilon}(s)\vert}\,\mathrm{d}x\,\mathrm{d}v\,\mathrm{d}s,
\end{align*}
Since $\vert v\vert \le 2\sqrt{h}_{\varepsilon}(s,x,v)$ and the $L^p$-norms are conserved as \eqref{eq-Lp-conse}, we have
\begin{align*}
&\int_0^t\iint_{\vert x-\xi_{\alpha,\varepsilon}(s)\vert \le\sigma,h_{\varepsilon}(s,x,v)\le R} \frac{f_{\varepsilon}(s,x,v)}{\vert x-\xi_{\alpha,\varepsilon}(s)\vert}\,\mathrm{d}x\,\mathrm{d}v\,\mathrm{d}s\\
&\le \|f_{\varepsilon}^0\|_{\infty}\int_0^t\iint_{\vert x-\xi_{\alpha,\varepsilon}(s)\vert \le\sigma,\vert v\vert \le 2\sqrt{R}} \frac{1}{\vert x-\xi_{\alpha,\varepsilon}(s)\vert}\,\mathrm{d}x\,\mathrm{d}v\,\mathrm{d}s\\
&\le 8\pi^2\|f_{\varepsilon}^0\|_{\infty}R\sigma t.
\end{align*}
Hence we have
\begin{equation}\label{eq-2.10}
L_{0,\varepsilon}(t)\le N\sigma^{-1}t+NR^{-1/2}L_{1,\varepsilon}(t)+8\pi^2N\|f_{\varepsilon}^0\|_{\infty}R\sigma t.
\end{equation}
Take $k=1$, $l=2$ in Proposition~\ref{prn-estimate-Lk}, we have
\begin{align*}
L_{1,\varepsilon}(t)&\le C\Big((\varepsilon+t\varepsilon^{-1}) H_{2,\varepsilon}(t)+\|f_{\varepsilon}^0\|_{\infty}^{\frac{1}{4}}H_{2,\varepsilon}(t)^{\frac{3}{4}}\int_0^t\|E_{\varepsilon}(s)\|_{4}\,\mathrm{d}s\nonumber\\
&\qquad+L_{0,\varepsilon}(t)+\big(L_{0,\varepsilon}(t)+t\big)H_{1,\varepsilon}(t)\Big).
\end{align*}
By the weak Young's inequality, the interpolation inequality in $L^p$  and \eqref{eq-priori-2}, we have
\begin{align*}
\|E_{\varepsilon}(s)\|_{4}\le C\|\rho_{\varepsilon}(s)\|_1^{\frac{1}{2}}\|\rho_{\varepsilon}(s)\|_2^{\frac{1}{2}}\le C\|f_{\varepsilon}^{0}\|_{\infty}^{\frac{1}{4}},
\end{align*}
combining with $H_{2,\varepsilon}(t)\le C$ by Lemma~\ref{lem-interpo-Hk}, we obtain
\begin{align}\label{eq-2.11}
L_{1,\varepsilon}(t)\le\tilde{C}\Big(t\varepsilon^{-1}+\varepsilon+L_{0,\varepsilon}(t)+\|f_{\varepsilon}^0\|_{\infty}^{\frac{1}{2}}t\Big).
\end{align}
for some $\tilde{C}>1$.

Now take $R=\max\{(2N\tilde{C})^2,t^{-1}\}$ and insert \eqref{eq-2.11} into \eqref{eq-2.10}, we have
\begin{align*}
L_{0,\varepsilon}(t)&\le N\sigma^{-1}t+\min\{1/2,\tilde{C}Nt^{1/2}\}\Big(t\varepsilon^{-1}+\varepsilon+L_{0,\varepsilon}(t)+\|f_{\varepsilon}^0\|_{\infty}^{\frac{1}{2}}t\Big)\\
&\quad+8\pi^2N\|f_{\varepsilon}^0\|_{\infty}\max\{(2N\tilde{C})^2,t^{-1}\}\sigma t.
\end{align*}
Take $\sigma=\varepsilon t^{1/2}$ and recall the fact $\|f_{\varepsilon}^0\|_{\infty}\le C\varepsilon^{-2}$ by the assumptions \eqref{eq-initial-assumptions},  we obtain
\begin{align*}
L_{0,\varepsilon}(t)&\le 2N\varepsilon^{-1}t^{1/2}+\min\{1,2\tilde{C}Nt^{1/2}\}\Big(\varepsilon+(C+1)\varepsilon^{-1}t\Big)\\
&\quad+16\pi^2NC\varepsilon^{-1}\max\{(2N\tilde{C})^2,t^{-1}\}t^{3/2}.
\end{align*}
Let $t_0=(2N\tilde{C})^{-2}$, we obtain \eqref{eq-coro-key-L0-2}.

Similarly, take $R=\max\{(2N\tilde{C})^2,1\}$,  $\sigma=\varepsilon$ and insert \eqref{eq-2.11} into \eqref{eq-2.10}, we obtain \eqref{eq-coro-key-L0-1}.
\end{proof}

\section{Time regularity} \label{sec-time-regularity}

In this section, we obtain the time equi-continuity of $\rho_{\varepsilon}$ and $\xi_{\alpha,\varepsilon}$, which will be used in the compactness argument in the next section.

First of all, according to the ODE of $\eta_{\alpha,\varepsilon}$ in \eqref{eq-VP-point}, a consequence of \eqref{eq-coro-key-L0-2} combing with the boundedness of $\eta_{\alpha,\varepsilon}$ by \eqref{eq-priori-1} is that for all $0<s<t<T$
\begin{equation}\label{eq-eta-continuity}
\vert\eta_{\alpha,\varepsilon}(t)-\eta_{\alpha,\varepsilon}(s)\vert\le C(t-s)^{1/2}\varepsilon^{-2}.
\end{equation}

\subsection{Weak formulation}

As in \cite{GS99,Mio19}, we reexpress the system \eqref{eq-VP-point} using a weak formulation in order to study the asymptotical behavior as $\varepsilon\to 0$.

\begin{Lemma}\label{lem-weak-form-rho}
$\varepsilon\int vf_{\varepsilon}\,\mathrm{d}v$ is bounded in $C^{1/4}(\mathbb{R}_+,W^{-1,1}(\mathbb{R}^2))$ and the following equation holds in the distributional sense:
\begin{equation*}
\partial_t\rho_{\varepsilon}+\nabla_x\cdot\Big((E_{\varepsilon}+F_{\varepsilon})^{\perp}\rho_{\varepsilon}\Big)
=\nabla_x\cdot\Big(\nabla_x \cdot\int_{\mathbb{R}^2}v \otimes v f_{\varepsilon}\,\mathrm{d}v\Big)^{\perp}
+\varepsilon\partial_t\nabla_x\cdot\int_{\mathbb{R}^2}v^{\perp}f_{\varepsilon}\,\mathrm{d}v.
\end{equation*}
\end{Lemma}
\begin{proof}
From Corollary~\ref{coro-key-L0}, $F_{\varepsilon}\rho_{\varepsilon}\in L_{t,x,v}^1$ for each $\varepsilon>0$. Hence testing \eqref{eq-VP-point} and applying the dominated convergence theorem, we can obtain that the following equations hold in the distributional sense (see \cite[Lem.3.2]{GS99} for details):
\begin{align*}
&\partial_t\rho_{\varepsilon}+\varepsilon^{-1}\nabla_x\cdot\int vf_{\varepsilon}\,\mathrm{d}v=0,\label{eq-}\\
&\varepsilon\partial_t\int vf_{\varepsilon}\,\mathrm{d}v+\nabla_x\cdot\int_{\mathbb{R}^2}v \otimes v f_{\varepsilon}\,\mathrm{d}v-(E_{\varepsilon}+F_{\varepsilon})\rho_{\varepsilon}-\varepsilon^{-1}\int v^{\perp}f_{\varepsilon}\,\mathrm{d}v=0,
\end{align*}
which implies the equation in the lemma. Moreover, the second equation implies: for all $0\le s<t<T$ with $t-s<t_0$
\begin{align*}
&\varepsilon\left\|\int vf_{\varepsilon}(t)\,\mathrm{d}v-\int vf_{\varepsilon}(s)\,\mathrm{d}v\right\|_{W^{-1,1}(\mathbb{R}^2)}\\
&\le C(t-s)\sup_{\tau\ge 0}\left[\varepsilon^{-1}\iint(1+\vert v\vert^2)f_{\varepsilon}(\tau)\,\mathrm{d}v\,\mathrm{d}x+\|\widetilde{E}_{\varepsilon}(\tau)\|_{2}\|\rho_{\varepsilon}(\tau)\|_{2}+\|\bar{E}_{\varepsilon}(\tau)\|_{\infty}\|\rho_{\varepsilon}(\tau)\|_{1}\right]\\
&\quad+C\varepsilon^{-1}(t-s)^{1/2}\\
&\le C\varepsilon^{-1}(t-s)^{1/2}.
\end{align*}
where we have used \eqref{eq-coro-key-L0-2} and the bounds in Sec.~\ref{subsec-LH-1-bound-rho}. On the other hand
\begin{equation*}
\varepsilon\left\|\int vf_{\varepsilon}(t)\,\mathrm{d}v-\int vf_{\varepsilon}(s)\,\mathrm{d}v\right\|_{W^{-1,1}(\mathbb{R}^2)}\le C\varepsilon\sup_{\tau\ge 0}\iint\vert v\vert f_{\varepsilon}(\tau)\,\mathrm{d}v\,\mathrm{d}x\le C\varepsilon.
\end{equation*}
Hence
\begin{equation*}
\varepsilon\left\|\int vf_{\varepsilon}(t)\,\mathrm{d}v-\int vf_{\varepsilon}(s)\,\mathrm{d}v\right\|_{W^{-1,1}(\mathbb{R}^2)}\le C(t-s)^{1/4}.
\end{equation*}
\end{proof}

Now we further organize the weak formulation in Lemma~\ref{lem-weak-form-rho} using the symmetric quadratic form $\mathcal{H}_{\Phi}$, which allows us to handle the singular term $F_{\varepsilon}^{\perp}\rho_{\varepsilon}$. The proof is analogous to \cite[Prn.2.10]{Mio19}, but with simplification and improved estimate of the remainder term.
\begin{Proposition}\label{prn-weak-form}
Let $\Phi\in C_{c}^{\infty}(\mathbb{R}_{+}\times\mathbb{R}^2)$, for all $t\ge 0$ we have
\begin{align*}
&\int_{\mathbb{R}^2}\Phi(t,x)\big(\rho_{\varepsilon}(t,x)\,\mathrm{d}x+\bar{\delta}_{\varepsilon}(t,\mathrm{d}x)\big)-\int_{\mathbb{R}^2}\Phi(0,x)\big(\rho_{\varepsilon}(0,x)\,\mathrm{d}x+\bar{\delta}_{\varepsilon}(0,\mathrm{d}x)\big)\\
&=\int_0^t\int_{\mathbb{R}^2}\partial_t\Phi(s,x)\big(\rho_{\varepsilon}(s,x)\,\mathrm{d}x+\bar{\delta}_{\varepsilon}(s,\mathrm{d}x)\big)\,\mathrm{d}s+\int_0^t\mathcal{H}_{\Phi(s,\cdot)}[\rho_{\varepsilon}(s,\cdot)+\bar{\delta}_{\varepsilon}(s,\cdot),\rho_{\varepsilon}(s,\cdot)+\bar{\delta}_{\varepsilon}(s,\cdot)]\,\mathrm{d}s\\
&\quad+\int_0^t\int_{\mathbb{R}^2}\nabla_x^2\Phi(s,x):\int_{\mathbb{R}^2}v^{\perp}\otimes vf_{\varepsilon}(s,x,v)\,\mathrm{d}v\,\mathrm{d}s\,\mathrm{d}x\\
&\quad+\sum_{\alpha}\int_0^t\nabla_x^2\Phi(s,\xi_{\alpha,\varepsilon}(s)):\eta_{\alpha,\varepsilon}^{\perp}\otimes\eta_{\alpha,\varepsilon}(s)\,\mathrm{d}s+R_{\varepsilon}(t)
\end{align*}
where $\bar{\delta}_{\varepsilon}(t,x)=\sum_{\alpha}\delta_{\xi_{\alpha,\varepsilon}(t)}(x)$. The reminder $R_{\varepsilon}$ satisfies the following estimate: for all $0\le s<t\le T$
\begin{equation}\label{eq-R-varepsilon-est-2}
\vert R_{\varepsilon}(t)-R_{\varepsilon}(s)\vert\le C\sum_{i=0,1}\|\partial_t^i\nabla_x\Phi\|_{L^{\infty}(\mathbb{R}_{+}\times\mathbb{R}^2)}\min\{\varepsilon,(t-s)^{1/4}\}.
\end{equation}
\end{Proposition}
\begin{proof}
Apply the equation in Lemma~\ref{lem-weak-form-rho} with the test function $\Phi$. After symmetrizing the term $\nabla_x\cdot(E^{\perp}_{\varepsilon}\rho_{\varepsilon})$ as \eqref{eq-symmetric-form}, we obtain
\begin{align}
&\int_{\mathbb{R}^2}\Phi(t,x)\rho_{\varepsilon}(t,x)\,\mathrm{d}x-\int_{\mathbb{R}^2}\Phi(0,x)\rho_{\varepsilon}(0,x)\,\mathrm{d}x-\int_0^t\int_{\mathbb{R}^2}\partial_t\Phi(s,x)\rho_{\varepsilon}(s,x)\,\mathrm{d}s\,\mathrm{d}x\nonumber\\
&=\int_0^t\mathcal{H}_{\Phi(s,\cdot)}[\rho_{\varepsilon}(s),\rho_{\varepsilon}(s)]\,\mathrm{d}s+\int_0^t\int_{\mathbb{R}^2}\nabla_x\Phi(s,x)\cdot F_{\varepsilon}^{\perp}\rho_{\varepsilon}(s,x)\,\mathrm{d}s\,\mathrm{d}x\nonumber\\
&\quad+\int_0^t\int_{\mathbb{R}^2}\nabla_x^2\Phi(s,x):\int_{\mathbb{R}^2}v^{\perp}\otimes vf_{\varepsilon}(s,x,v)\,\mathrm{d}v\,\mathrm{d}s\,\mathrm{d}x
+I_{\varepsilon}^1(t),\label{eq-pf-prn-weak-form-1}
\end{align}
where
\begin{align*}
I_{\varepsilon}^1(t)&=\varepsilon\int_0^t\int_{\mathbb{R}^2}\partial_t\nabla_x\Phi(s,x)\cdot \int_{\mathbb{R}^2}v^{\perp}f_{\varepsilon}(s,x,v)\,\mathrm{d}v\,\mathrm{d}x\,\mathrm{d}s\\
&\quad-\varepsilon\int_{\mathbb{R}^2}\nabla_x\Phi(t,x)\cdot \int_{\mathbb{R}^2}v^{\perp}f_{\varepsilon}(t,x,v)\,\mathrm{d}v\,\mathrm{d}x
+\varepsilon\int_{\mathbb{R}^2}\nabla_x\Phi(0,x)\cdot \int_{\mathbb{R}^2}v^{\perp}f_{\varepsilon}^0(x,v)\,\mathrm{d}v\,\mathrm{d}x,
\end{align*}
by \eqref{eq-priori-1} and the continuity of $\varepsilon\int v^{\perp}f_{\varepsilon}\,\mathrm{d}v$ in Lemma~\ref{lem-weak-form-rho}, we have for all $0\le s<t\le T$
\begin{equation}\label{eq-prn-weak-form-R-1-1}
\vert I_{\varepsilon}^1(t)-I_{\varepsilon}^1(s)\vert\le C\sum_{i=0,1}\|\partial_t^i\nabla_x\Phi\|_{L^{\infty}(\mathbb{R}_{+}\times\mathbb{R}^2)}\min\{\varepsilon,(t-s)^{1/4}\}.
\end{equation}

Next by simple calculation, we have
\begin{align*}
&\Phi(t,\xi_{\alpha,\varepsilon}(t))-\Phi(0,\xi_{\alpha,\varepsilon}(0))\\
&=\int_0^t\partial_t\Phi(s,\xi_{\alpha,\varepsilon}(s))\,\mathrm{d}s+\varepsilon^{-1}\int_0^t\eta_{\alpha,\varepsilon}(s)\cdot\nabla_x\Phi(s,\xi_{\alpha,\varepsilon}(s))\,\mathrm{d}s
\end{align*}
and
\begin{align*}
&\varepsilon\nabla_x\Phi(t,\xi_{\alpha,\varepsilon}(t))\cdot\eta_{\alpha,\varepsilon}^{\perp}(t)-\varepsilon\nabla_x\Phi(0,\xi_{\alpha,\varepsilon}(0))\cdot\eta_{\alpha,\varepsilon}^{\perp}(0)\\
&=\varepsilon\int_0^t\partial_t\nabla_x\Phi(s,\xi_{\alpha,\varepsilon}(s))\cdot\eta_{\alpha,\varepsilon}^{\perp}(s)\,\mathrm{d}s+\int_0^t\nabla_x^2\Phi(s,\xi_{\alpha,\varepsilon}(s)):\eta_{\alpha,\varepsilon}^{\perp}\otimes\eta_{\alpha,\varepsilon}(s)\,\mathrm{d}s\\
&\quad+\int_0^t\nabla_x\Phi(s,\xi_{\alpha,\varepsilon}(s))\cdot\Big(-\frac{\eta_{\alpha,\varepsilon}(s)}{\varepsilon}+E^{\perp}_{\varepsilon}(s,\xi_{\alpha,\varepsilon}(s))+\sum_{\beta:\beta\ne\alpha}\frac{(\xi_{\alpha,\varepsilon}(s)-\xi_{\beta,\varepsilon}(s))^{\perp}}{\vert\xi_{\alpha,\varepsilon}(s)-\xi_{\beta,\varepsilon}(s)\vert^2}\Big)\,\mathrm{d}s.
\end{align*}
Adding the two equations above, we obtain
\begin{align}
&\Phi(t,\xi_{\alpha,\varepsilon}(t))-\Phi(0,\xi_{\alpha,\varepsilon}(0))\nonumber\\
&=\int_0^t\partial_t\Phi(s,\xi_{\alpha,\varepsilon}(s))\,\mathrm{d}s+\int_0^t\nabla_x^2\Phi(s,\xi_{\alpha,\varepsilon}(s)):\eta_{\alpha,\varepsilon}^{\perp}\otimes\eta_{\alpha,\varepsilon}(s)\,\mathrm{d}s\nonumber\\
&\quad+\int_0^t\nabla_x\Phi(s,\xi_{\alpha,\varepsilon}(s))\cdot\left(E^{\perp}_{\varepsilon}(s,\xi_{\alpha,\varepsilon}(s))+\sum_{\beta:\beta\ne\alpha}\frac{(\xi_{\alpha,\varepsilon}(s)-\xi_{\beta,\varepsilon}(s))^{\perp}}{\vert\xi_{\alpha,\varepsilon}(s)-\xi_{\beta,\varepsilon}(s)\vert^2}\right)\,\mathrm{d}s+I_{\varepsilon}^2(t),\label{eq-pf-prn-weak-form-2}
\end{align}
where
\begin{align*}
I_{\varepsilon}^2(t)=&\varepsilon\nabla_x\Phi(0,\xi_{\alpha,\varepsilon}(0))\cdot\eta_{\alpha,\varepsilon}^{\perp}(0)-\varepsilon\nabla_x\Phi(t,\xi_{\alpha,\varepsilon}(t))\cdot\eta_{\alpha,\varepsilon}^{\perp}(t)\\
&+\varepsilon\int_0^t\partial_t\nabla_x\Phi(s,\xi_{\alpha,\varepsilon}(s))\cdot\eta_{\alpha,\varepsilon}^{\perp}(s)\,\mathrm{d}s.
\end{align*}
by the estimates in Proposition~\ref{prn-priori} we know that
\begin{equation}\label{eq-prn-weak-form-R-2-1}
\vert I_{\varepsilon}^2(t)\vert\le C\varepsilon(t\|\partial_t\nabla_x\Phi(s,x)\|_{\infty}+\|\nabla_x\Phi(t,x)\|_{\infty}).
\end{equation}
Recall \eqref{eq-eta-continuity}, we have  for all $0\le s<t\le T$
\begin{equation}\label{eq-prn-weak-form-R-2-2}
\vert I_{\varepsilon}^2(t)-I_{\varepsilon}^2(s)\vert\le C\sum_{i=0,1}\|\partial_t^i\nabla_x\Phi\|_{L^{\infty}(\mathbb{R}_{+}\times\mathbb{R}^2)}\varepsilon^{-1}(t-s)^{1/2}.
\end{equation}

Moreover, we observe that
\begin{align*}
&\int_0^t\int_{\mathbb{R}^2}\nabla_x\Phi(s,x)F_{\varepsilon}^{\perp}(s,x)\rho_{\varepsilon}(s,x)\,\mathrm{d}s\,\mathrm{d}x+\sum_{\alpha}\int_0^t\nabla_x\Phi(s,\xi_{\alpha,\varepsilon}(s))\cdot E^{\perp}_{\varepsilon}(s,\xi_{\alpha,\varepsilon}(s))\,\mathrm{d}s\\
&=\sum_{\alpha}\int_0^t\int_{\mathbb{R}^2}\big(\nabla_x\Phi(s,x)-\nabla_x\Phi(s,\xi_{\alpha,\varepsilon}(s))\big)\cdot\frac{(x-\xi_{\alpha,\varepsilon}(s))^{\perp}}{\vert x-\xi_{\alpha,\varepsilon}(s)\vert^2}\rho_{\varepsilon}(s,x)\,\mathrm{d}x\,\mathrm{d}s\\
&=2\int_0^t\mathcal{H}_{\Phi(s)}[\rho_{\varepsilon}(s),\bar{\delta}_{\varepsilon}(s)]\,\mathrm{d}s.
\end{align*}
Similarly
\begin{equation*}
\sum_{\alpha}\int_0^t\nabla_x\Phi(s,\xi_{\alpha,\varepsilon}(s))\cdot\sum_{\beta:\beta\ne\alpha}\frac{(\xi_{\alpha,\varepsilon}(s)-\xi_{\beta,\varepsilon}(s))^{\perp}}{\vert\xi_{\alpha,\varepsilon}(s)-\xi_{\beta,\varepsilon}(s)\vert^2}\,\mathrm{d}s=\frac{1}{2}\sum_{\alpha\ne\beta}\int_0^tH_{\Phi(s)}\big(\xi_{\alpha,\varepsilon}(s),\xi_{\beta,\varepsilon}(s)\big)\,\mathrm{d}s.
\end{equation*}
And notice that
\begin{align*}
&\mathcal{H}_{\Phi(s)}[\rho_{\varepsilon}(s)+\bar{\delta}_{\varepsilon}(s),\rho_{\varepsilon}(s)+\bar{\delta}_{\varepsilon}(s)]\\
&=\mathcal{H}_{\Phi(s)}[\rho_{\varepsilon}(s),\rho_{\varepsilon}(s)]+2\mathcal{H}_{\Phi(s)}[\rho_{\varepsilon}(s),\bar{\delta}_{\varepsilon}(s)]+\frac{1}{2}\sum_{\alpha\ne\beta}H_{\Phi(s)}\big(\xi_{\alpha,\varepsilon}(s),\xi_{\beta,\varepsilon}(s)\big).
\end{align*}
By the above observation, the weak formulation follows by adding \eqref{eq-pf-prn-weak-form-1} and the sum in $\alpha$ of \eqref{eq-pf-prn-weak-form-2}. The estimate on the remainder term $R_{\varepsilon}=I_{\varepsilon}^1+I_{\varepsilon}^2$ follows from \eqref{eq-prn-weak-form-R-1-1}, \eqref{eq-prn-weak-form-R-2-1}, \eqref{eq-prn-weak-form-R-2-2}.
\end{proof}

\subsection{Time equi-continuity}

As explained in the introduction, due to the fact that the trajectories of the point charges cannot be separated from each other when the strength of the magnetic field tends towards infinity, Miot's method is ineffective to obtain the time equi-continuity of  $\rho_{\varepsilon}$ and $\xi_{\alpha,\varepsilon}$. In this subsection, we address this problem by an iteration argument.

Firstly, we obtain the following lemma by the arguments of \cite[Cor.2.11]{Mio19}.

\begin{Lemma}\label{lem-min-continuity-point-charge}
Let $T>0$. There exists  $\widetilde{K}_{0}>2$ depending only on $N,T,K_0,K_1$ and let $\widetilde{\varepsilon}_0:=\widetilde{K}_{0}^{-16}$, such that for all $0<\varepsilon <\widetilde{\varepsilon}_0$ and for all  $0\le s < t\le T$ with $t-s\le\widetilde{\varepsilon}_0$
\begin{equation*}
\max_{\alpha}\min_{\beta}\vert\xi_{\alpha,\varepsilon}(t)-\xi_{\beta,\varepsilon}(s)\vert\le\widetilde{K}_{0}\min\{\sqrt{t-s+\varepsilon},(t-s)^{1/4}\}.
\end{equation*}
\end{Lemma}
\begin{Remark}\label{rem-4-1}
By the above lemma, for each $\alpha$, there exists $\beta_{\alpha}$ such that
\begin{equation*}
\vert\xi_{\alpha,\varepsilon}(t)-\xi_{\beta_{\alpha},\varepsilon}(s)\vert\le\widetilde{K}_{0}\min\{\sqrt{t-s+\varepsilon},(t-s)^{1/4}\}.
\end{equation*}
Since any two point charges have a strictly positive distance according to \eqref{eq-priori-positive-distance} and $\varepsilon$ is small enough, $\beta_{\alpha}\ne\beta_{\alpha'}$ if $\alpha\ne\alpha'$. Hence
\begin{align*}
\left\vert\sum_{\alpha}\xi_{\alpha,\varepsilon}(t)-\sum_{\beta}\xi_{\beta,\varepsilon}(s)\right\vert&\le\sum_{\alpha}\vert\xi_{\alpha,\varepsilon}(t)-\xi_{\beta_{\alpha},\varepsilon}(s)\vert\\
&\le N\widetilde{K}_{0}\min\{\sqrt{t-s+\varepsilon},(t-s)^{1/4}\},
\end{align*}
which means that the trajectory of the center of the point charges is time equi-continuous.
\end{Remark}
\begin{proof}
We define $g(t,s,\varepsilon):=\widetilde{K}_{0}\min\{\sqrt{t-s+\varepsilon},(t-s)^{1/4}\}$ with $\widetilde{K}_0$ determined later.  By Proposition~\ref{prn-weak-form} we have for all  $\Phi\in C_{c}^{\infty}(\mathbb{R}^2)$ and for all  $0\le s < t\le T$
\begin{align*}
&\int_{\mathbb{R}^2}\Phi(x)\big(\rho_{\varepsilon}(t,x)\,\mathrm{d}x+\bar{\delta}_{\varepsilon}(t,\mathrm{d}x)\big)-\int_{\mathbb{R}^2}\Phi(x)\big(\rho_{\varepsilon}(s,x)\,\mathrm{d}x+\bar{\delta}_{\varepsilon}(s,\mathrm{d}x)\big)\\
&=\int_s^t\mathcal{H}_{\Phi}[\rho_{\varepsilon}(\tau,\cdot)+\bar{\delta}_{\varepsilon}(\tau,\cdot),\rho_{\varepsilon}(\tau,\cdot)+\bar{\delta}_{\varepsilon}(\tau,\cdot)]\,\mathrm{d}\tau+\int_s^t\iint_{\mathbb{R}^4}\nabla_x^2\Phi(x):v^{\perp}\otimes vf_{\varepsilon}(\tau,x,v)\,\mathrm{d}v\,\mathrm{d}x\,\mathrm{d}\tau\\
&\quad+\sum_{\alpha}\int_s^t\nabla_x^2\Phi(\xi_{\alpha,\varepsilon}(\tau)):\eta_{\alpha,\varepsilon}^{\perp}\otimes\eta_{\alpha,\varepsilon}(\tau)\,\mathrm{d}\tau+R_{\varepsilon}(t)-R_{\varepsilon}(s).
\end{align*}
Since $\vert H_{\Phi}(x,y)\vert\le C\|\nabla_x^2\Phi\|_{\infty}$, by Proposition~\ref{prn-priori} and \eqref{eq-R-varepsilon-est-2}, it follows that
\begin{align}
&\Big\vert\int_{\mathbb{R}^2}\Phi(x)\rho_{\varepsilon}(t,x)\,\mathrm{d}x+\sum_{\alpha}\Phi(\xi_{\alpha,\varepsilon}(t))-\int_{\mathbb{R}^2}\Phi(x)\rho_{\varepsilon}(s,x)\,\mathrm{d}x-\sum_{\alpha}\Phi(\xi_{\alpha,\varepsilon}(s))\Big\vert\nonumber\\
&\le C\|\nabla_x^2\Phi\|_{\infty}(t-s)\sup_{\tau\in[s,t]}\Big(\|\rho_{\varepsilon}(\tau)\|^{2}_{1}+\|\rho_{\varepsilon}(\tau)\|_{1}+\iint_{\mathbb{R}^2}\vert v\vert^{2}f_{\varepsilon}(\tau)\,\mathrm{d}x\,\mathrm{d}v+\sum_{\alpha}\vert\eta_{\alpha,\varepsilon}(\tau)\vert^{2}\Big)\nonumber\\
&\quad+C\|\nabla_x\Phi\|_{\infty}\min\{\varepsilon,(t-s)^{1/4}\}\nonumber\\
&\le C\|\nabla_x^2\Phi\|_{\infty}(t-s)+C\|\nabla_x\Phi\|_{\infty}\min\{\varepsilon,(t-s)^{1/4}\}.\label{eq-lem-min-continuity-point-charge-1}
\end{align}

By contradiction, we assume there exist $\bar{\varepsilon},\bar{t},\bar{s}$,  without loss of generality assume $\bar{t}>\bar{s}$, and some $\bar{\alpha}$ satisfying $\min_{\beta}\vert\xi_{\bar{\alpha},\bar{\varepsilon}}(\bar{t})-\xi_{\beta,\bar{\varepsilon}}(\bar{s})\vert>g(\bar{t},\bar{s},\bar{\varepsilon})$. We set
\begin{equation*}
\Phi(x)=\chi\left(\frac{2(x-\xi_{\bar{\alpha},\bar{\varepsilon}}(\bar{t}))}{g(\bar{t},\bar{s},\bar{\varepsilon})}\right),
\end{equation*}
where $\chi$ is a smooth cut-off function such that $\chi=1$ on $B(0,1)$, $\chi$ vanishes on $B(0,2)^c$ and $0\le \chi(x)\le 1$ for all $x\in\mathbb{R}^2$. In particular, we have
\begin{equation*}
\sum_{\alpha}\Phi(\xi_{\alpha,\bar{\varepsilon}}(\bar{t}))\ge 1,\quad\sum_{\alpha}\Phi(\xi_{\alpha,\bar{\varepsilon}}(\bar{s}))=0,\quad\|\nabla_x\Phi\|_{\infty}\le \frac{C}{g(\bar{t},\bar{s},\bar{\varepsilon})},\quad\|\nabla_x^2\Phi\|_{\infty}\le \frac{C}{[g(\bar{t},\bar{s},\bar{\varepsilon})]^2}.
\end{equation*}
Moreover, by the definition of $g$, we have
\begin{align*}
&\|\nabla_x^2\Phi\|_{\infty}(\bar{t}-\bar{s})\le\frac{C(\bar{t}-\bar{s})}{\widetilde{K}_{0}^2\min\{\bar{t}-\bar{s}+\bar{\varepsilon},(\bar{t}-\bar{s})^{1/2}\}}\le C\widetilde{K}_0^{-2},\\
&\|\nabla_x\Phi\|_{\infty}\min\{\bar{\varepsilon},(\bar{t}-\bar{s})^{1/4}\}\le\frac{C\min\{\bar{\varepsilon},(\bar{t}-\bar{s})^{1/4}\}}{\widetilde{K}_{0}\min\{\sqrt{\bar{t}-\bar{s}+\bar{\varepsilon}},(\bar{t}-\bar{s})^{1/4}\}}\le C\widetilde{K}_0^{-1}.
\end{align*}
Notice also $g(\bar{t},\bar{s},\bar{\varepsilon})\le\widetilde{K}_{0}^{-1}$ by the definition of $\widetilde{\varepsilon}_0$. In view of \eqref{eq-lem-min-continuity-point-charge-1} with $t=\bar{t},s=\bar{s}$ and using \eqref{eq-estimate-non-concentration}, we get
\begin{align*}
1&\le 2\sup_{\tau\in[0,T]}\int_{B(\xi_{\bar{\alpha},\bar{\varepsilon}}(\tau),g(\bar{t},\bar{s},\bar{\varepsilon}))}\rho_{\bar{\varepsilon}}(\tau,x)\,\mathrm{d}x+C\widetilde{K}_0^{-1}\\
&\le 2\sup_{\tau\in[0,T]}\int_{B(\xi_{\bar{\alpha},\bar{\varepsilon}}(\tau),\widetilde{K}_0^{-1})}\rho_{\bar{\varepsilon}}(\tau,x)\,\mathrm{d}x+C\widetilde{K}_0^{-1}\\
&\le C\left(\vert\ln\widetilde{K}_0\vert^{-1/2}+\widetilde{K}_0^{-1}\right)\le\frac{1}{2}.
\end{align*}
if we choose $\widetilde{K}_0$ depending only on $N,T,K_0,K_1$ sufficiently large. This yields a contradiction and we obtain $\max_{\alpha}\min_{\beta}\vert\xi_{\alpha,\varepsilon}(t)-\xi_{\beta,\varepsilon}(s)\vert\le g(t,s,\varepsilon)$.
\end{proof}

Now we prove the crucial lemma in this section, which allows us to obtain the time equi-continuity of each point charges from Lemma~\ref{lem-min-continuity-point-charge}.

\begin{Lemma}\label{lem-small-time-continuity-point-charge}
Choose $\widetilde{K}_{0}>2$ large enough in Lemma~\ref{lem-min-continuity-point-charge} and let $\widetilde{\varepsilon}_0=\widetilde{K}_{0}^{-16}$. Then for all $0<\varepsilon <\widetilde{\varepsilon}_0$, $0\le s < t\le T$ with $t-s\le \widetilde{\varepsilon}_0$ and $1\le\alpha\le N$
\begin{equation*}
\vert\xi_{\alpha,\varepsilon}(t)-\xi_{\alpha,\varepsilon}(s)\vert\le\widetilde{K}_{0}\min\{\sqrt{t-s+\varepsilon},(t-s)^{1/4}\}.
\end{equation*}
\end{Lemma}
\begin{proof}
We define $g(t,s,\varepsilon):=\widetilde{K}_{0}\min\{\sqrt{t-s+\varepsilon},(t-s)^{1/4}\}$. Firstly, recall the equation of $\xi_{\alpha,\varepsilon}$ in \eqref{eq-VP-point} and the uniform bound of $\vert\eta_{\alpha,\varepsilon}\vert$ in \eqref{eq-priori-1}, we have for all $\alpha$ and $0\le s<t\le T$ with $t-s\le \widetilde{\varepsilon}_0\varepsilon$
\begin{equation}\label{eq-pf-lem44-1}
\vert\xi_{\alpha,\varepsilon}(t)-\xi_{\alpha,\varepsilon}(s)\vert\le \frac{\widetilde{C}(t-s)}{\varepsilon}\le \widetilde{C}\widetilde{\varepsilon}_0.
\end{equation}
Recall that $M$ is the constant given in \eqref{eq-priori-positive-distance}, we choose $\widetilde{K}_0$ large enough such that
\begin{equation}\label{eq-choice-tilde-K0}
\max\{4\widetilde{C}\widetilde{\varepsilon}_0,16\widetilde{K}_0^{-1}\}<M.
\end{equation}
On the one hand, by \eqref{eq-priori-positive-distance}, \eqref{eq-pf-lem44-1}  and \eqref{eq-choice-tilde-K0}, we have for all $t-s\le \widetilde{\varepsilon}_0\varepsilon$, $\alpha\ne\beta$
\begin{equation}\label{eq-pf-lem44-2}
\vert\xi_{\alpha,\varepsilon}(t)-\xi_{\beta,\varepsilon}(s)\vert\ge\vert\xi_{\alpha,\varepsilon}(s)-\xi_{\beta,\varepsilon}(s)\vert-\vert\xi_{\alpha,\varepsilon}(t)-\xi_{\alpha,\varepsilon}(s)\vert\ge\frac{3M}{4}.
\end{equation}
On the other hand, by Lemma~\ref{lem-min-continuity-point-charge} and \eqref{eq-choice-tilde-K0}, we have for all $t-s\le \widetilde{\varepsilon}_0$ and for each $\alpha=1,\dots,N$,
\begin{equation}\label{eq-pf-lem44-3}
\min_{\beta}\vert\xi_{\alpha,\varepsilon}(t)-\xi_{\beta,\varepsilon}(s)\vert\le g(t,s,\varepsilon)\le \widetilde{K}_0^{-1}<M/2.
\end{equation}
\eqref{eq-pf-lem44-2} together with \eqref{eq-pf-lem44-3} implies that: for all $0<\varepsilon<\widetilde{\varepsilon}_0$ and $0\le s<t\le T$ with $t-s\le \widetilde{\varepsilon}_0\varepsilon$
\begin{equation}\label{eq-pf-lem44-4}
\vert\xi_{\alpha,\varepsilon}(t)-\xi_{\alpha,\varepsilon}(s)\vert\le g(t,s,\varepsilon).
\end{equation}
Now we need to iterate this estimate to obtain that it holds for $t-s\le \Delta_{k}(\varepsilon)$, where $\Delta_{k}(\varepsilon)$ is a sequence satisfying $\Delta_{k}(\varepsilon)<\Delta_{k+1}(\varepsilon)$ and $\lim_{k\to\infty}\Delta_{k}(\varepsilon)=\widetilde{\varepsilon}_0$ for all $\varepsilon>0$.  The iteration mainly consists of two steps:
\begin{itemize}
\item \eqref{eq-pf-lem44-4} holds for $t-s\le \Delta_{k}(\varepsilon)$ $\Longrightarrow$ \eqref{eq-pf-lem44-2} holds for $t-s\le \Delta_{k+1}(\varepsilon)$.
\item \eqref{eq-pf-lem44-2} for $t-s\le \Delta_{k+1}(\varepsilon)$ together with \eqref{eq-pf-lem44-3} for $t-s\le \widetilde{\varepsilon}_0$ $\Longrightarrow$ \eqref{eq-pf-lem44-4} holds for $t-s\le \Delta_{k+1}(\varepsilon)$.
\end{itemize}

Concretely, we define $\Delta_{k}=\Delta_{k+1}(\varepsilon):=\widetilde{\varepsilon}_0\varepsilon^{m^{-k}}$ with $m=\frac{4}{3}$. Assume there exists $k\ge 0$ such that \eqref{eq-pf-lem44-4} holds for $t-s\le \Delta_{k}$, the iteration is constructed as follows.

\noindent{\bf Step 1.} let $t-s\le \Delta_{k+1}$, we define $t_0=s$, $t_i=s+i\Delta_{k}$ for $i\le n_k:=\lfloor\Delta_{k+1}/ \Delta_{k}\rfloor$, $t_{n_k+1}=t$. By \eqref{eq-pf-lem44-4} we have
\begin{align*}
\vert\xi_{\alpha,\varepsilon}(t)-\xi_{\alpha,\varepsilon}(s)\vert&\le\sum_{i=0}^{n_k}g(t_{i+1},t_i,\varepsilon)\\
&\le (n_k+1)\widetilde{K}_{0}\sqrt{\Delta_{k}+\varepsilon}\\
&\le\left(\varepsilon^{m^{-(k+1)}-m^{-k}}+1\right)\widetilde{K}_{0}\sqrt{\widetilde{\varepsilon}_0\varepsilon^{m^{-k}}+\varepsilon}\\
&\le\left(\varepsilon^{m^{-(k+1)}-\frac{m^{-k}}{2}}+\varepsilon^{\frac{m^{-k}}{2}}\right)2\widetilde{K}_{0}\widetilde{\varepsilon}_0^{\frac{1-m^{-k}}{2}}.
\end{align*}
\begin{itemize}
\item If $k=0$, then by the assumptions $m=\frac{4}{3}$, $\widetilde{\varepsilon}_0=\widetilde{K}_0^{-16}$ and \eqref{eq-choice-tilde-K0}, we have
\begin{align*}
\vert\xi_{\alpha,\varepsilon}(t)-\xi_{\alpha,\varepsilon}(s)\vert&\le\left(\varepsilon^{\frac{1}{m}-\frac{1}{2}}+\varepsilon^{\frac{1}{2}}\right)2\widetilde{K}_{0}\le 4\widetilde{K}_{0}^{-1}<M/4.
\end{align*}
\item If $k\ge 1$, then by the assumptions $m=\frac{4}{3}$, $\widetilde{\varepsilon}_0=\widetilde{K}_0^{-16}$ and \eqref{eq-choice-tilde-K0}, we have
\begin{align*}
\vert\xi_{\alpha,\varepsilon}(t)-\xi_{\alpha,\varepsilon}(s)\vert&\le 4\widetilde{K}_{0}\widetilde{\varepsilon}_0^{\frac{1-m^{-k}}{2}}\le 4\widetilde{K}_{0}^{-1}<M/4.
\end{align*}
\end{itemize}
Hence we have $\vert\xi_{\alpha,\varepsilon}(t)-\xi_{\alpha,\varepsilon}(s)\vert<M/4$ holds for $t-s\le \Delta_{k+1}$.  By \eqref{eq-priori-positive-distance} and the triangle inequality, we have that \eqref{eq-pf-lem44-2} holds for $t-s\le \Delta_{k+1}$. 

\noindent{\bf Step 2.} By the conclusion of Step 1, \eqref{eq-pf-lem44-2} holds for $t-s\le \Delta_{k+1}$. Together with \eqref{eq-pf-lem44-3} for $t-s\le \widetilde{\varepsilon}_0$, we have immediately that \eqref{eq-pf-lem44-4} holds for $t-s\le \Delta_{k+1}$. 

Since \eqref{eq-pf-lem44-4} holds for $k=0$, the iteration can begin. By induction, we deduce that: for all $0<\varepsilon<\widetilde{\varepsilon}_0$, $k\ge0$ and $0\le s<t\le T$ with $t-s\le \widetilde{\varepsilon}_0\varepsilon^{m^{-k}}$
\begin{equation*}
\vert\xi_{\alpha,\varepsilon}(t)-\xi_{\alpha,\varepsilon}(s)\vert\le g(t,s,\varepsilon).
\end{equation*}
Notice $\varepsilon^{m^{-k}}\to 1$ as $k\to\infty$, hence the inequlaity holds also for $t-s\le \widetilde{\varepsilon}_0$. The proof is complete.
\end{proof}

Based on Lemma~\ref{lem-small-time-continuity-point-charge}, we obtain finally the time equi-continuity of $\rho_{\varepsilon}$ and $\xi_{\alpha,\varepsilon}$.
\begin{Proposition}\label{prn-continuity-point-charge}
Let $T>0$. There exists $\widetilde{K}> 1$ and $\widetilde{\varepsilon}_0>0$ depending only on $N,T,K_0,K_1$, such that for all $0< \varepsilon < \widetilde{\varepsilon}_0$ and for all  $0\le s < t\le T$
\begin{equation*}
\|\rho_{\varepsilon}(t)-\rho_{\varepsilon}(s)\|_{\dot{W}^{-2,1}}+\sum_{\alpha}\vert\xi_{\alpha,\varepsilon}(t)-\xi_{\alpha,\varepsilon}(s)\vert\le\widetilde{K}\min\{\sqrt{t-s+\varepsilon},(t-s)^{1/4}\}.
\end{equation*}
\end{Proposition}
\begin{proof}
Applying Lemma~\ref{lem-small-time-continuity-point-charge} for $t-s\le\widetilde{\varepsilon}_0$ to obtain 
\begin{equation*}
\sum_{\alpha}\vert\xi_{\alpha,\varepsilon}(t)-\xi_{\alpha,\varepsilon}(s)\vert\le N\widetilde{K}_0\min\{\sqrt{t-s+\varepsilon},(t-s)^{1/4}\}.
\end{equation*}
For $t-s>\widetilde{\varepsilon}_0$, notice $\xi_{\alpha,\varepsilon}$ is uniformly bounded by \eqref{eq-priori-1}, we have 
\begin{equation*}
\sum_{\alpha}\vert\xi_{\alpha,\varepsilon}(t)-\xi_{\alpha,\varepsilon}(s)\vert\le C\le C[\min(\widetilde{\varepsilon}_0^{1/2},\widetilde{\varepsilon}_0^{1/4})]^{-1}\min\{\sqrt{t-s+\varepsilon},(t-s)^{1/4}\}.
\end{equation*}
Hence, denote $\widetilde{K}_1=N\widetilde{K}_0+ C[\min(\widetilde{\varepsilon}_0^{1/2},\widetilde{\varepsilon}_0^{1/4})]^{-1}$, we have
\begin{equation}\label{eq-pf-prn-45-1}
\sum_{\alpha}\vert\xi_{\alpha,\varepsilon}(t)-\xi_{\alpha,\varepsilon}(s)\vert\le \widetilde{K}_1\min\{\sqrt{t-s+\varepsilon},(t-s)^{1/4}\}\quad \text{for all } 0\le s < t\le T.
\end{equation}

It follows from \eqref{eq-lem-min-continuity-point-charge-1} and \eqref{eq-pf-prn-45-1} that
\begin{align*}
&\Big\vert\int_{\mathbb{R}^2}\Phi(x)\rho_{\varepsilon}(t,x)\,\mathrm{d}x-\int_{\mathbb{R}^2}\Phi(x)\rho_{\varepsilon}(s,x)\,\mathrm{d}x\Big\vert\\
&\le\sum_{\alpha}\vert\Phi(\xi_{\alpha,\varepsilon}(t))-\Phi(\xi_{\alpha,\varepsilon}(s))\vert+C\|\nabla_x^2\Phi\|_{\infty}(t-s)+C\|\nabla_x\Phi\|_{\infty}\min\{\varepsilon,(t-s)^{1/4}\}\\
&\le\|\Phi\|_{\dot{W}^{2,\infty}}\left(\sum_{\alpha}\vert\xi_{\alpha,\varepsilon}(t)-\xi_{\alpha,\varepsilon}(s)\vert+C\sum_{i=1,2}[\tilde{g}(t,s,\varepsilon)]^i\right)\\
&\le\|\Phi\|_{\dot{W}^{2,\infty}}\left(\widetilde{K}_1\tilde{g}(t,s,\varepsilon)+C\sum_{i=1,2}[\tilde{g}(t,s,\varepsilon)]^i\right),
\end{align*}
where $\tilde{g}(t,s,\varepsilon)=\min\{\sqrt{t-s+\varepsilon},(t-s)^{1/4}\}$. Hence
\begin{equation}\label{eq-pf-prn-45-2}
\|\rho_{\varepsilon}(t)-\rho_{\varepsilon}(s)\|_{\dot{W}^{-2,1}}\le\left(\widetilde{K}_1+C+CT^{1/4}\right)\min\{\sqrt{t-s+\varepsilon},(t-s)^{1/4}\}\quad \text{for all } 0\le s < t\le T.
\end{equation}
Adding \eqref{eq-pf-prn-45-1} and \eqref{eq-pf-prn-45-2}, the conclusion follows with $\widetilde{K}=2\widetilde{K}_1+C+CT^{1/4}$.
\end{proof}

\section{Proof of Theorem~\ref{thm-main}}\label{sec-proof-of-main-thm}

We define the radial moment density as
\begin{equation*}
m_{k,\varepsilon}(t,x,\omega)=\int_{0}^{\infty}r^kf_{\varepsilon}(t,x,r\omega)\,\mathrm{d}r.
\end{equation*}
Notice the term in Proposition~\ref{prn-weak-form}
\begin{equation*}
\int_0^t\int_{\mathbb{R}^2}\nabla_x^2\Phi(s,x):\int_{\mathbb{R}^2}v^{\perp}\otimes vf_{\varepsilon}(s,x,v)\,\mathrm{d}v\,\mathrm{d}s\,\mathrm{d}x=\int_0^t\int_{\mathbb{R}^2}\int_{\mathbb{S}^1}\widetilde{\Phi}(s,x,\omega)m_{3,\varepsilon}(s,x,\omega)\,\mathrm{d}\omega\,\mathrm{d}s\,\mathrm{d}x,
\end{equation*}
where $\widetilde{\Phi}(s,x,\omega)=\nabla_x^2\Phi(s,x):\omega^{\perp}\otimes \omega\in C_c^{\infty}(\mathbb{R}_+\times\mathbb{R}^2\times\mathbb{S}^1)$.

Now we obtain the limit points by the compactness arguments stated in \cite{GS99,Mio19}.

{\bf Convergence of $f_{\varepsilon}$:} By the $L^1$-bound \eqref{eq-Lp-conse}, there exist a subsequence such that $f_{\varepsilon_n}\to f$ weak-* in $L^{\infty}(\mathbb{R}_{+},\mathcal{M}_+(\mathbb{R}^2\times\mathbb{R}^2))$. Moreover, $\nabla_{v}\cdot(v^{\perp}f)=0$ in the distributional sense by the lemmas in \cite[Lem.3.3]{GS99}, \cite[Lem.2.16]{Mio19}, which implies
\begin{equation*}
\int_0^t\int_{\mathbb{R}^2}\int_{\mathbb{S}^1}\widetilde{\Phi}(s,x,\omega)\,m_{3}(s,\mathrm{d}x,\mathrm{d}\omega)\,\mathrm{d}s=\int_0^t\int_{\mathbb{R}^2}\nabla_x^2\Phi(s,x):\int_{\mathbb{R}^2}v^{\perp}\otimes v\,f(s,\mathrm{d}x,\mathrm{d}v)\,\mathrm{d}s\equiv0,
\end{equation*}
where the measure $m_3(t)$ is defined by
\begin{equation*}
\iint_{\mathbb{R}^2\times\mathbb{S}^1}\Psi(x,\omega)\,m_{3}(t,\mathrm{d}x,\mathrm{d}\omega)=\iint_{\mathbb{R}^2\times\mathbb{R}^2}\Psi(x,v/\vert v\vert)\vert v\vert^2\,f(t,\mathrm{d}x,\mathrm{d}v)\quad\forall\Psi\in C_0(\mathbb{R}^2\times\mathbb{S}^1).
\end{equation*}

{\bf Convergence of $\rho_{\varepsilon}$:} Notice by the uniform bound of $\iint\vert v\vert^2f_{\varepsilon}$ in \eqref{eq-priori-1}, we have $\rho_{\varepsilon_n}\to \rho=\int_{\mathbb{R}^2}\,f(\mathrm{d}v)$ weak-* in $L^{\infty}(\mathbb{R}_{+},\mathcal{M}_+(\mathbb{R}^2))$. Recall Proposition~\ref{prn-continuity-point-charge}, $\rho\in C^{1/2}([0,T],W^{-2,1}(\mathbb{R}^2))$ for all $T>0$.

{\bf Convergence of $\xi_{\alpha,\varepsilon}$:} Recall Proposition~\ref{prn-continuity-point-charge} again and by the Arzel\`a-Ascoli theorem, there exists a subsequence, still denoted as $\xi_{\alpha,\varepsilon_n}$, such that $\xi_{\alpha,\varepsilon_n}\to\xi_{\alpha}$ in $C^{r}([0,T],\mathbb{R}^2)$, $\forall r\in[0,1/4)$ and $\xi_{\alpha}\in C^{1/2}([0,T],\mathbb{R}^2)$, $\forall T>0$, which implies
\begin{equation*}
\bar{\delta}_{\varepsilon_n}=\sum_{\alpha}\delta_{\xi_{\alpha,\varepsilon_n}}\to\bar{\delta}=\sum_{\alpha}\delta_{\xi_{\alpha}}\text{ weak-* in }L^{\infty}(\mathbb{R}_{+},\mathcal{M}_+(\mathbb{R}^2)).
\end{equation*}

{\bf Convergence of $\eta_{\alpha,\varepsilon}$:} Notice by \eqref{eq-priori-1} again, $\vert\eta_{\alpha,\varepsilon_n}\vert^2$ is uniformly bounded in $L^{\infty}(\mathbb{R}_{+})$. Hence there exists a subsequence such that
\begin{equation*}
\eta_{\alpha,\varepsilon_n}^{\perp}\otimes\eta_{\alpha,\varepsilon_n}\to \mathbf{M}^{\alpha}\text{ weak-* in }L^{\infty}(\mathbb{R}_{+},\mathbb{R}^{2\times 2}),
\end{equation*}
where $\mathbf{M}^{\alpha}$ satisfies $\mathbf{M}_{11}^{\alpha}=-\mathbf{M}_{22}^{\alpha}$ and $\mathbf{M}_{12}^{\alpha}=-\mathbf{M}_{21}^{\alpha}$.

{\bf The existence of defect measure:} Notice by \eqref{eq-priori-1}
\begin{equation*}
\sup_{t\ge0,\varepsilon>0}\int_{\mathbb{R}^2}\int_{\mathbb{S}^1}m_{3,\varepsilon}(s,x,\omega)\,\mathrm{d}\omega\,\mathrm{d}x=\sup_{t\ge0,\varepsilon>0}\iint_{\mathbb{R}^2\times\mathbb{R}^2}\vert v\vert^2f_{\varepsilon}(s,x,v)\,\mathrm{d}v\,\mathrm{d}x\le C,
\end{equation*}
there exist a subsequence, still denoted as $m_{3,\varepsilon_n}$, such that $m_{3,\varepsilon_n}\to \mu$ weak-* in $L^{\infty}(\mathbb{R}_{+},\mathcal{M}(\mathbb{R}^2\times\mathbb{S}^1))$. Denote $\mu_{0}:=\mu-m_3$, then $\mu_{0}\in L^{\infty}(\mathbb{R}_{+},\mathcal{M}_+(\mathbb{R}^2\times\mathbb{S}^1))$ by the arguments of \cite[p.802, Proof of Thm.A]{GS99}.

The compactness statements above imply that 
\begin{align*}
&\int_{\mathbb{R}^2}\Phi(t,x)\big(\rho_{\varepsilon_n}(t,x)\,\mathrm{d}x+\bar{\delta}_{\varepsilon_n}(t,\mathrm{d}x)\big)-\int_{\mathbb{R}^2}\Phi(0,x)\big(\rho_{\varepsilon_n}(0,x)\,\mathrm{d}x+\bar{\delta}_{\varepsilon_n}(0,\mathrm{d}x)\big)\\
&-\int_0^t\int_{\mathbb{R}^2}\partial_t\Phi(s,x)\big(\rho_{\varepsilon_n}(s,x)\,\mathrm{d}x+\bar{\delta}_{\varepsilon_n}(s,\mathrm{d}x)\big)\,\mathrm{d}s-\int_0^t\int_{\mathbb{R}^2}\nabla_x^2\Phi(s,x):\int_{\mathbb{R}^2}v^{\perp}\otimes vf_{\varepsilon_n}(s,x,v)\,\mathrm{d}v\,\mathrm{d}s\,\mathrm{d}x\\
&\quad-\sum_{\alpha}\int_0^t\nabla_x^2\Phi(s,\xi_{\alpha,\varepsilon}(s)):\eta_{\alpha,\varepsilon}^{\perp}\otimes\eta_{\alpha,\varepsilon}(s)\,\mathrm{d}s-R_{\varepsilon_n}(t)
\end{align*}
converges to
\begin{align*}
&\int_{\mathbb{R}^2}\Phi(t,x)\big(\rho(t,\mathrm{d}x)+\bar{\delta}(t,\mathrm{d}x)\big)-\int_{\mathbb{R}^2}\Phi(0,x)\big(\rho(0,\mathrm{d}x)+\bar{\delta}(0,\mathrm{d}x)\big)\\
&-\int_0^t\int_{\mathbb{R}^2}\partial_t\Phi(s,x)\big(\rho(s,\mathrm{d}x)+\bar{\delta}(s,\mathrm{d}x)\big)\,\mathrm{d}s-\int_0^t\int_{\mathbb{R}^2}\nabla_x^2\Phi(s,x):\int_{\mathbb{S}^1}\omega^{\perp}\otimes\omega\,\mu_0(s,\mathrm{d}x,\mathrm{d}\omega)\,\mathrm{d}s\\
&\quad-\sum_{\alpha}\int_0^t\nabla_x^2\Phi(s,\xi_{\alpha}(s)):\mathbf{M}^{\alpha}(s)\,\mathrm{d}s.
\end{align*}

Finally, we prove
\begin{align*}
\int_0^t\mathcal{H}_{\Phi(s,\cdot)}[\rho_{\varepsilon_n}(s,\cdot)+\bar{\delta}_{\varepsilon_n}(s,\cdot),\rho_{\varepsilon_n}(s,\cdot)+\bar{\delta}_{\varepsilon_n}(s,\cdot)]\,\mathrm{d}s
\end{align*}
converges to
\begin{align*}
\int_0^t\mathcal{H}_{\Phi(s,\cdot)}[\rho(s,\cdot)+\bar{\delta}(s,\cdot),\rho(s,\cdot)+\bar{\delta}(s,\cdot)]\,\mathrm{d}s
\end{align*}
by the following lemma, which has been established in \cite{Del91,Sch95,Mio19}.
\begin{Lemma}\label{lem-convergence-quadratic-form}
Let $\mu_n^1,\mu_n^2$ be sequences in $L^{\infty}(\mathbb{R}_+,\mathcal{M}_+(\mathbb{R}^2))$ such that $\mu_n^1,\mu_n^2\to\mu^1,\mu^2$ weak-* in $L^{\infty}(\mathbb{R}_+,\mathcal{M}_+(\mathbb{R}^2))$, respectively. Assume $\mu_n^1,\mu_n^2$ are equicontinuous in time with values in some negative Sobolev space $W^{-s,q}(\mathbb{R}^2)$ and satisfy the non-concentration condition 
\begin{align*}
&\lim_{\epsilon\to 0^+}\sup_{0\le t\le T}\sup_{\vert x_0\vert\le R_1}\sup_n\mu_n^1(t,B(x_0,\epsilon R_2))=0\quad\forall T,R_1,R_2>0,\\
&\sup_{t\ge 0}\iint_{\vert x-y\vert\le R}\vert H_{\Phi(t)}(x,y)\vert\mu_n^2(t,\mathrm{d}x)\mu_n^2(t,\mathrm{d}y)=0\quad\forall\Phi\in C_c^{\infty}(\mathbb{R}_+\times\mathbb{R}^2),\text{ for some }R>0.
\end{align*}
Assume also $\mu^1,\mu^2$ satisfy the non-concentration condition.

Then $\int_0^t\mathcal{H}_{\Phi(s)}[\mu_n^1(s)+\mu_n^2(s),\mu_n^1(s)+\mu_n^2(s)]\,\mathrm{d}s$ converges to $\int_0^t\mathcal{H}_{\Phi(s)}[\mu^1(s)+\mu^2(s),\mu^1(s)+\mu^2(s)]\,\mathrm{d}s$ for all $t\ge 0$.
\end{Lemma}

Let $\mu_n^1=\rho_{\varepsilon_n}$, $\mu_n^2=\bar{\delta}_{\varepsilon_n}$, then the time equicontinuity holds by Proposition~\ref{prn-continuity-point-charge}. The non-concentration conditions hold by \eqref{eq-estimate-non-concentration} and \eqref{eq-priori-positive-distance}. Hence the convergence stated above holds and the theorem is proved.

\begin{proof}[Proof of Lemma~\ref{lem-convergence-quadratic-form}]
It is proved directly by following the arguments between \cite[Lem.3.2]{Sch95} and \cite[Thm.3.3]{Sch95}. Indeed, note $H_{\Phi}(x,y)$ is smooth outside $x=y$, we split it as
\begin{equation*}
H_{\Phi}(x,y)=H_{\Phi}(x,y)\psi_{\epsilon}(\vert x-y\vert)+H_{\Phi}(x,y)[1-\psi_{\epsilon}(\vert x-y\vert)]=:H_{\Phi}^{\epsilon}(x,y)+\tilde{H}_{\Phi}^{\epsilon}(x,y)
\end{equation*}
where $\psi\in C_c^{\infty}(\mathbb{R}_{+},[0,1])$ such  that $\psi$ vanishes on  $[0,1]$ and  $\psi=1$ on $[2,+\infty)$ and set $\psi_{\epsilon}= \psi(\cdot/\epsilon)$. Then $H_{\Phi}^{\epsilon}(x,y)$ is smooth for all $x,y\in\mathbb{R}^2$ as long as $\epsilon>0$ and the support of $\tilde{H}_{\Phi}^{\epsilon}$ is contained in $\{(x,y)\in\mathbb{R}^4:\vert x-y\vert\le 2\epsilon\}$.

By the definition of $\mathcal{H}_{\Phi(s)}[\cdot,\cdot]$, we have
\begin{align*}
&\int_0^t\mathcal{H}_{\Phi(s)}[\mu_n^1(s)+\mu_n^2(s),\mu_n^1(s)+\mu_n^2(s)]\,\mathrm{d}s\\
&=\frac{1}{2}\int_0^t\iint_{\mathbb{R}^4}H_{\Phi}^{\epsilon}(x,y)[\mu_n^1(s,\mathrm{d}x)+\mu_n^2(s,\mathrm{d}x)][\mu_n^1(s,\mathrm{d}y)+\mu_n^2(s,\mathrm{d}y)]\,\mathrm{d}s,\\
&\quad+\frac{1}{2}\int_0^t\iint_{\mathbb{R}^4}\tilde{H}_{\Phi}^{\epsilon}(x,y)[\mu_n^1(s,\mathrm{d}x)+\mu_n^2(s,\mathrm{d}x)][\mu_n^1(s,\mathrm{d}y)+\mu_n^2(s,\mathrm{d}y)]\,\mathrm{d}s\\
&=:I_1^{\epsilon,n}+I_2^{\epsilon,n}.
\end{align*}

One the one hand, by \cite[Lem.3.2]{Sch95}, $\mu_n^1\mu_n^1,\mu_n^1\mu_n^2,\mu_n^2\mu_n^2\to\mu^1\mu^1,\mu^1\mu^2,\mu^2\mu^2$ weak-* in $L^{\infty}(\mathbb{R}_+,\mathcal{M}(\mathbb{R}^2))$, respectively, hence
\begin{align*}
\lim_{n\to\infty}I_1^{\epsilon,n}&=\frac{1}{2}\int_0^t\iint_{\mathbb{R}^4}H_{\Phi}^{\epsilon}(x,y)[\mu^1(s,\mathrm{d}x)+\mu^2(s,\mathrm{d}x)][\mu^1(s,\mathrm{d}y)+\mu^2(s,\mathrm{d}y)]\,\mathrm{d}s\\
&=\frac{1}{2}\int_0^t\iint_{\mathbb{R}^4}H_{\Phi}(x,y)[\mu^1(s,\mathrm{d}x)+\mu^2(s,\mathrm{d}x)][\mu^1(s,\mathrm{d}y)+\mu^2(s,\mathrm{d}y)]\,\mathrm{d}s-\tilde{I}_1^{\epsilon},
\end{align*}
where
\begin{equation*}
\tilde{I}_1^{\epsilon}=\frac{1}{2}\int_0^t\iint_{\mathbb{R}^4}\tilde{H}_{\Phi}^{\epsilon}(x,y)[\mu^1(s,\mathrm{d}x)+\mu^2(s,\mathrm{d}x)][\mu^1(s,\mathrm{d}y)+\mu^2(s,\mathrm{d}y)]\,\mathrm{d}s.
\end{equation*}
On the other hand, without loss of generality assuming the diameter of the support in $x$ of $\Phi$ is $R_1>0$, then the support of $\tilde{H}_{\Phi}^{\epsilon}$ is contained in $\{(x,y)\in\mathbb{R}^4:\vert x-y\vert\le 2\epsilon,\,\vert x\vert\le R_1+2\epsilon,\,\vert y\vert\le R_1+2\epsilon\}$, one has
\begin{align*}
\vert I_2^{\epsilon,n}\vert\le&\frac{1}{2}\int_0^t\iint_{\mathbb{R}^4}\vert\tilde{H}_{\Phi}^{\epsilon}(x,y)\vert\mu_n^1(s,\mathrm{d}x)\mu_n^1(s,\mathrm{d}y)\,\mathrm{d}s+\int_0^t\iint_{\mathbb{R}^4}\vert\tilde{H}_{\Phi}^{\epsilon}(x,y)\vert\mu_n^1(s,\mathrm{d}x)\mu_n^2(s,\mathrm{d}y)\,\mathrm{d}s\\
&+\frac{1}{2}\int_0^t\iint_{\vert x-y\vert\le 2\epsilon}\vert H_{\Phi}(x,y)\vert\mu_n^2(s,\mathrm{d}x)\mu_n^2(s,\mathrm{d}y)\,\mathrm{d}s\\
\le& t\|H_{\Phi}\|_{\infty}(\|\mu_n^1\|_{\mathcal{M}}+\|\mu_n^2\|_{\mathcal{M}})\sup_{0\le s\le t}\sup_{\vert x\vert\le R_1+2\epsilon}\sup_n\vert\mu_n^1\vert(s,B(x,2\epsilon))\\
&+t\sup_{0\le s\le t}\iint_{\vert x-y\vert\le 2\epsilon}\vert H_{\Phi(s)}(x,y)\vert\mu_n^2(s,\mathrm{d}x)\mu_n^2(s,\mathrm{d}y).
\end{align*}
Hence by the non-concentration condition, we have $\lim_{\epsilon\to 0}\sup_{n}\vert I_2^{\epsilon,n}\vert=0$, similarly $\lim_{\epsilon\to 0}\vert \tilde{I}_1^{\epsilon}\vert=0$, which finishes the proof.

\end{proof}

\section{Appendix}

\subsection{Proof of Proposition~\ref{prn-E0nu-to-E00}}

Let $\psi\in C_c^{\infty}(\mathbb{R}_{+},[0,1])$ such  that $\psi$ vanishes on  $[0,1]$ and  $\psi=1$ on $[2,+\infty)$ and set $\psi_{\epsilon}= \psi(\cdot/\epsilon)$, which converges to  $1$ almost  everywhere. Let $\Phi$ be a test function and define
\begin{align*}
&\Phi_{\epsilon}(t,x):=\Phi(t,x)\prod_{\alpha}\psi_{\epsilon}(\vert x-\xi_{\alpha}(t)\vert),\\
&\Phi_{\epsilon}^{\alpha}(t,x):=\Phi(t,x)\prod_{\beta:\beta\ne\alpha}\psi_{\epsilon}(\vert x-\xi_{\alpha}(t)\vert).
\end{align*}
Then it can be proved without difficulty by the fact that the estimate \eqref{eq-priori-positive-distance} holds for $\{\xi_{\alpha}\}$: for $\epsilon$ small enough, we have
\begin{align*}
&\Phi_{\epsilon}(t,\xi_{\alpha}(t))=0, \partial_{t}\Phi_{\epsilon}(t,\xi_{\alpha}(t))=0, \nabla\Phi_{\epsilon}(t,\xi_{\alpha}(t))=0\text{ for all }1\le\alpha\le N,\\ 
&\Phi_{\epsilon}^{\alpha}(t,\xi_{\beta}(t))=0, \partial_{t}\Phi_{\epsilon}^{\alpha}(t,\xi_{\beta}(t))=0, \nabla\Phi_{\epsilon}^{\alpha}(t,\xi_{\beta}(t))=0\text{ for all }1\le\beta\ne\alpha\le N.
\end{align*}

Notice $E\in L^{\infty}_{\rm{loc}}(L^{\infty})$ and $\frac{\rho}{\vert x-\xi_{\alpha}\vert}\in L^{1}_{\rm{loc}}$ by the assumption $\rho\in L^{\infty}_{\rm{loc}}L^{p}$ for $p>2$. Hence take $\Phi_{\epsilon}$ as test functions in \eqref{eq-weak-solution-mv-Euler-defect} with $\nu=0$, we obtain
\begin{align*}
&\frac{\mathrm{d}}{\mathrm{d}t}\int_{\mathbb{R}^2}\Phi_{\epsilon}(t,x)\rho(t,x)\,\mathrm{d}x
\\
&=\int_{\mathbb{R}^2}\partial_{t}\Phi_{\epsilon}(t,x)\rho(t,x)\,\mathrm{d}x+\int_{\mathbb{R}^2}\left(E^{\perp}(t,x)+\sum_{\alpha}\frac{(x-\xi_{\alpha})^{\perp}}{\vert x-\xi_{\alpha}\vert^2}\right)\cdot\nabla\Phi_{\epsilon}(t,x)\rho(t,x)\,\mathrm{d}x.
\end{align*}
We claim that the equation above converges to the first equation in \eqref{eq-Vortex-Wave} as $\epsilon\to0$.

A similar routine can prove that
\begin{align*}
&\frac{\mathrm{d}}{\mathrm{d}t}\int_{\mathbb{R}^2}\Phi_{\epsilon}^{\alpha}(t,x)\rho(t,x)\,\mathrm{d}x+\frac{\mathrm{d}}{\mathrm{d}t}\Phi_{\epsilon}^{\alpha}(t,\xi_{\alpha}(t))
\\
&=\int_{\mathbb{R}^2}\partial_{t}\Phi_{\epsilon}^{\alpha}(t,x)\rho(t,x)\,\mathrm{d}x+\partial_{t}\Phi_{\epsilon}^{\alpha}(t,\xi_{\alpha}(t))+\int_{\mathbb{R}^2}\left(E^{\perp}(t,x)+\sum_{\alpha}\frac{(x-\xi_{\alpha})^{\perp}}{\vert x-\xi_{\alpha}\vert^2}\right)\cdot\nabla\Phi_{\epsilon}^{\alpha}(t,x)\rho(t,x)\,\mathrm{d}x\\
&\quad+E^{\perp}(t,\xi_{\alpha}(t))\cdot\nabla\Phi_{\epsilon}^{\alpha}(t,\xi_{\alpha}(t))+\sum_{\alpha\ne\beta}\frac{(\xi_{\alpha}-\xi_{\beta})^{\perp}}{\vert\xi_{\alpha}-\xi_{\beta}\vert^2}\cdot\nabla\Phi_{\epsilon}^{\alpha}(t,\xi_{\alpha}(t))
\end{align*}
converges to
\begin{align*}
&\frac{\mathrm{d}}{\mathrm{d}t}\int_{\mathbb{R}^2}\Phi(t,x)\rho(t,x)\,\mathrm{d}x+\frac{\mathrm{d}}{\mathrm{d}t}\Phi(t,\xi_{\alpha}(t))
\\
&=\int_{\mathbb{R}^2}\partial_{t}\Phi(t,x)\rho(t,x)\,\mathrm{d}x+\partial_{t}\Phi(t,\xi_{\alpha}(t))+\int_{\mathbb{R}^2}\left(E^{\perp}(t,x)+\sum_{\alpha}\frac{(x-\xi_{\alpha})^{\perp}}{\vert x-\xi_{\alpha}\vert^2}\right)\cdot\nabla\Phi(t,x)\rho(t,x)\,\mathrm{d}x\\
&\quad+E^{\perp}(t,\xi_{\alpha}(t))\cdot\nabla\Phi(t,\xi_{\alpha}(t))+\sum_{\alpha\ne\beta}\frac{(\xi_{\alpha}-\xi_{\beta})^{\perp}}{\vert\xi_{\alpha}-\xi_{\beta}\vert^2}\cdot\nabla\Phi(t,\xi_{\alpha}(t)),
\end{align*}
which minus the first equation in \eqref{eq-Vortex-Wave} to obtain the second equation in \eqref{eq-Vortex-Wave}.

{\bf Proof of the claim.} By Lebesgue's dominated convergence theorem, $\int\Phi_{\epsilon}(t,x)\rho(t,x)\,\mathrm{d}x$ tends to  $\int\Phi(t,x)\rho(t,x)\,\mathrm{d}x$ as $\epsilon\to 0$.

By direct computation, we have
\begin{align*}
\int_{\mathbb{R}^2}\partial_{t}\Phi_{\epsilon}(t,x)\rho(t,x)\,\mathrm{d}x&=\int_{\mathbb{R}^2}\partial_{t}\Phi(t,x)\prod_{\alpha}\psi_{\epsilon}(\vert x-\xi_{\alpha}\vert)\rho(t,x)\,\mathrm{d}x\\
&\quad+\int_{\mathbb{R}^2}\frac{\Phi(t,x)}{\epsilon}\sum_{\alpha}\Big[\frac{\xi_{\alpha}-x}{\vert \xi_{\alpha}-x\vert}\cdot\dot{\xi}_{\alpha}\psi'(\vert x-\xi_{\alpha}\vert/\epsilon)\prod_{\beta:\beta\ne\alpha}\psi_{\epsilon}(\vert x-\xi_{\beta}\vert)\Big]\rho(t,x)\,\mathrm{d}x\\
&=:I_{\epsilon}^1+I_{\epsilon}^2.
\end{align*}
It is obvious that $I_{\epsilon}^1$ converges to $\int\partial_{t}\Phi(t,x)\rho(t,x)\,\mathrm{d}x$  as $\epsilon\to 0$. For $I_{\epsilon}^2$, since by the assumption $\|\dot{\xi}_{\alpha}\|_{\infty}\le C$, we have by the H\"older inequality
\begin{equation*}
\vert I_{\epsilon}^2\vert\le\frac{C}{\epsilon}\sum_{\alpha}\int_{\vert x-\xi_{\alpha}(t)\vert\le2\epsilon} \vert \rho(t,x)\vert\,\mathrm{d}x\le\frac{C}{\epsilon}\|\rho(t)\|_{p}\epsilon^{2-\frac{2}{p}}.
\end{equation*}

Now we prove the convergence of the nonlinear term. Notice
\begin{align*}
&\int_{\mathbb{R}^2}\Big(E^{\perp}(t,x)+\sum_{\alpha}\frac{(x-\xi_{\alpha}(t))^{\perp}}{\vert x-\xi_{\alpha}(t)\vert^{2}}\Big)\cdot\nabla\Phi_{\epsilon}(t,x)\rho(t,x)\,\mathrm{d}x\\
&=\int_{\mathbb{R}^2}\Big(E^{\perp}(t,x)+\sum_{\alpha}\frac{(x-\xi_{\alpha}(t))^{\perp}}{|x-\xi_{\alpha}(t)|^{2}}\Big)\cdot\nabla\Phi(t,x)\prod_{\alpha}\psi_{\epsilon}(\vert x-\xi_{\alpha}(t)\vert)\rho(t,x)\,\mathrm{d}x\\
&+\int_{\mathbb{R}^2}E^{\perp}(t,x)\cdot\nabla\Big(\prod_{\alpha}\psi_{\epsilon}(\vert x-\xi_{\alpha}(t)\vert)\Big)\Phi(t,x)\rho(t,x)\,\mathrm{d}x\\
&+\int_{\mathbb{R}^2}\Big(\sum_{\alpha}\frac{(x-\xi_{\alpha}(t))^{\perp}}{\vert x-\xi_{\alpha}(t)\vert^{2}}\Big)\cdot\nabla\Big(\prod_{\alpha}\psi_{\epsilon}(\vert x-\xi_{\alpha}(t)\vert)\Big)\Phi(t,x)\rho(t,x)\,\mathrm{d}x\\
&=:I_{\epsilon}^3+I_{\epsilon}^4+I_{\epsilon}^5.
\end{align*}
On the one hand, the dominated convergence  theorem implies that  $I_{\epsilon}^3$ converges to
\begin{align*}
\int_{\mathbb{R}^2}\Big(E^{\perp}(t,x)+\sum_{\alpha}\frac{(x-\xi_{\alpha}(t))^{\perp}}{\vert x-\xi_{\alpha}(t)\vert^{2}}\Big)\cdot\nabla\Phi(t,x)\rho(t,x)\,\mathrm{d}x
\end{align*}
as $\epsilon\to 0$. On the other hand, we have by  H\"older inequality
\begin{align*}
\vert I_{\epsilon}^4\vert
\le\frac{C}{\epsilon}\sum_{\alpha}\int_{\vert x-\xi_{\alpha}(t)\vert\le2\epsilon}\vert E(t,x)\vert \vert\rho(t,x)\vert\,\mathrm{d}x\le\frac{C}{\epsilon}\|E(t)\|_{\infty}\|\rho(t)\|_{p}\epsilon^{2-\frac{2}{p}}.
\end{align*}

Notice by the identity $a^{\perp}\cdot a =0$, we have
\begin{align*}
I_{\epsilon}^5=\int_{\mathbb{R}^2}\frac{1}{\epsilon}\sum_{\alpha\ne\beta}\Big[\frac{(x-\xi_{\alpha})^{\perp}}{\vert x-\xi_{\alpha}\vert^{2}}\cdot\frac{x-\xi_{\beta}}{\vert x-\xi_{\beta}\vert}\psi'(\vert x-\xi_{\beta}\vert/\epsilon)\prod_{\gamma:\gamma\ne\beta}\psi_{\epsilon}(\vert x-\xi_{\gamma}\vert)\Big]\Phi\rho(t,x)\,\mathrm{d}x.
\end{align*}
When $\vert x-\xi_{\beta}(t)\vert\le2\epsilon$, from $\vert\xi_{\alpha}(t)-\xi_{\beta}(t)\vert\ge M$, we have $\vert x-\xi_{\alpha}(t)\vert\ge M-2\epsilon$ for $\alpha\ne\beta$, therefore we have
\begin{align*}
\vert I_{\epsilon}^5\vert\le\frac{C}{\epsilon}\sum_{\beta}\int_{\vert x-\xi_{\beta}(t)\vert\le2\epsilon} \vert \rho(t,x)\vert\,\mathrm{d}x\le\frac{C}{\epsilon}\|\rho(t)\|_{p}\epsilon^{2-\frac{2}{p}}.
\end{align*}
Since $1-2/p>0$, we have that $I_{\epsilon}^i$ with $i=2,4,5$ vanish in the limit $\epsilon\to 0$. Therefore, we have proved that $\rho$ satisfies the first equation of \eqref{eq-Vortex-Wave} in the sense of distributions.

\subsection{Proof of Proposition~\ref{prn-estimate-Lk}}
Firstly, we prove two estimates along the characteristic.
\begin{Lemma}\label{lem-estimate-character}
Assume $\vert X_{\varepsilon}(s)-\xi_{\alpha,\varepsilon}(s)\vert\le\frac{M}{2}$. Then we have
\begin{align*}
\frac{1}{\vert X_{\varepsilon}(s)-\xi_{\alpha,\varepsilon}(s)\vert}&\le\varepsilon^2\frac{\mathrm{d}^{2}}{\mathrm{d}s^{2}}\vert X_{\varepsilon}(s)-\xi_{\alpha,\varepsilon}(s)\vert+\varepsilon^{-1}\vert V_{\varepsilon}(s)-\eta_{\alpha,\varepsilon}(s)\vert\\
&\quad+\vert E_{\varepsilon}(s,X_{\varepsilon}(s))\vert+\vert E_{\varepsilon}(s,\xi_{\alpha,\varepsilon}(s))\vert+\frac{3N}{M}.
\end{align*}
\end{Lemma}
\begin{proof}
By direct calculation, we have
\begin{equation*}
\frac{\mathrm{d}}{\mathrm{d}s}\vert X_{\varepsilon}(s)-\xi_{\alpha,\varepsilon}(s)\vert=\frac{X_{\varepsilon}(s)-\xi_{\alpha,\varepsilon}(s)}{\vert X_{\varepsilon}(s)-\xi_{\alpha,\varepsilon}(s)\vert}\cdot\frac{V_{\varepsilon}(s)-\eta_{\alpha,\varepsilon}(s)}{\varepsilon},
\end{equation*}
and
\begin{align*}
\frac{\mathrm{d}^{2}}{\mathrm{d}s^{2}}\vert X_{\varepsilon}(s)-\xi_{\alpha,\varepsilon}(s)\vert&=\frac{\vert V_{\varepsilon}(s)-\eta_{\alpha,\varepsilon}(s)\vert^{2}}{\varepsilon^{2}\vert X_{\varepsilon}(s)-\xi_{\alpha,\varepsilon}(s)\vert}-\frac{[( X_{\varepsilon}(s)-\xi_{\alpha,\varepsilon}(s))\cdot(V_{\varepsilon}(s)-\eta_{\alpha,\varepsilon}(s))]^{2}}{\varepsilon^{2}\vert X_{\varepsilon}(s)-\xi_{\alpha,\varepsilon}(s)\vert^{3}}\nonumber\\
&\quad+\frac{ X_{\varepsilon}(s)-\xi_{\alpha,\varepsilon}(s)}{\varepsilon^2\vert X_{\varepsilon}(s)-\xi_{\alpha,\varepsilon}(s)\vert}\cdot\Big[\frac{V_{\varepsilon}^{\perp}(s)-\eta_{\alpha,\varepsilon}^{\perp}(s)}{\varepsilon}+(E_{\varepsilon}+F_{\varepsilon})(s,X_{\varepsilon}(s))\nonumber\\
&\qquad-E_{\varepsilon}(s,\xi_{\alpha,\varepsilon}(s))-\sum_{\beta:\beta\ne\alpha}\frac{\xi_{\alpha,\varepsilon}(s)-\xi_{\beta,\varepsilon}(s)}{\vert\xi_{\alpha,\varepsilon}(s)-\xi_{\beta,\varepsilon}(s)\vert^2}\Big]\nonumber\\
&\ge-\varepsilon^{-3}\vert V_{\varepsilon}(s)-\eta_{\alpha,\varepsilon}(s)\vert-\varepsilon^{-2}\vert E_{\varepsilon}(s,X_{\varepsilon}(s))\vert-\varepsilon^{-2}\vert E_{\varepsilon}(s,\xi_{\alpha,\varepsilon}(s))\vert\nonumber\\
&\quad-\sum_{\beta:\beta\ne\alpha}\frac{\varepsilon^{-2}}{\vert\xi_{\alpha,\varepsilon}(s)-\xi_{\beta,\varepsilon}(s)\vert}+\varepsilon^{-2}\frac{X_{\varepsilon}(s)-\xi_{\alpha,\varepsilon}(s)}{\vert X_{\varepsilon}(s)-\xi_{\alpha,\varepsilon}(s)\vert}\cdot F_{\varepsilon}(s,X_{\varepsilon}(s)),
\end{align*}
By Proposition~\ref{prn-priori}, we have
\begin{align*}
&\frac{X_{\varepsilon}(s)-\xi_{\alpha,\varepsilon}(s)}{\vert X_{\varepsilon}(s)-\xi_{\alpha,\varepsilon}(s)\vert}\cdot F_{\varepsilon}(s,X_{\varepsilon}(s))\\
&\le\varepsilon^2\frac{\mathrm{d}^{2}}{\mathrm{d}s^{2}}\vert X_{\varepsilon}(s)-\xi_{\alpha,\varepsilon}(s)\vert+\varepsilon^{-1}\vert V_{\varepsilon}(s)-\eta_{\alpha,\varepsilon}(s)\vert+\vert E_{\varepsilon}(s,X_{\varepsilon}(s))\vert+\vert E_{\varepsilon}(s,\xi_{\alpha,\varepsilon}(s))\vert+\frac{N}{M},
\end{align*}
which implies
\begin{align*}
\frac{1}{\vert X_{\varepsilon}(s)-\xi_{\alpha,\varepsilon}(s)\vert}&\le\sum_{\beta:\beta\ne\alpha}\frac{1}{\vert X_{\varepsilon}(s)-\xi_{\beta,\varepsilon}(s)\vert}+\varepsilon^2\frac{\mathrm{d}^{2}}{\mathrm{d}s^{2}}\vert X_{\varepsilon}(s)-\xi_{\alpha,\varepsilon}(s)\vert+\varepsilon^{-1}\vert V_{\varepsilon}(s)-\eta_{\alpha,\varepsilon}(s)\vert\\
&\quad+\vert E_{\varepsilon}(s,X_{\varepsilon}(s))\vert+\vert E_{\varepsilon}(s,\xi_{\alpha,\varepsilon}(s))\vert+\frac{N}{M}.
\end{align*}
By \eqref{eq-priori-positive-distance} and the assumption in the lemma, we have $\vert X_{\varepsilon}(s)-\xi_{\beta,\varepsilon}(s)\vert\ge M/2$ if $\beta\ne\alpha$, then the lemma follows.
\end{proof}
\begin{Lemma}\label{lem-estimate-pointwise-energy}
We have
\begin{align*}
\Big\vert\frac{\mathrm{d}}{\mathrm{d}s}\mathbf{h}_{\varepsilon}(s)\Big\vert
&\le C\varepsilon^{-1}\left(\sqrt{\mathbf{h}_{\varepsilon}(s)}(\vert E_{\varepsilon}(s,X_{\varepsilon}(s))\vert+1)+\sum_{\alpha}\frac{1}{\vert X_{\varepsilon}-\xi_{\alpha,\varepsilon}\vert}\right).
\end{align*}
\end{Lemma}
\begin{proof}
Differentiating $\mathbf{h}_{\varepsilon}(s)$ to obtain
\begin{align*}
\frac{\mathrm{d}}{\mathrm{d}s}\mathbf{h}_{\varepsilon}(s)&=\varepsilon^{-1}V_{\varepsilon}\cdot(E_{\varepsilon}+F_{\varepsilon})(s,X_{\varepsilon})+\sum_{\alpha}\left(\frac{X_{\varepsilon}-\xi_{\alpha,\varepsilon}}{\vert X_{\varepsilon}-\xi_{\alpha,\varepsilon}\vert}-\frac{X_{\varepsilon}-\xi_{\alpha,\varepsilon}}{\vert X_{\varepsilon}-\xi_{\alpha,\varepsilon}\vert^2}\right)\cdot\frac{V_{\varepsilon}-\eta_{\alpha,\varepsilon}}{\varepsilon}\\
&=\varepsilon^{-1}V_{\varepsilon}\cdot E_{\varepsilon}(s,X_{\varepsilon})+\sum_{\alpha}\left(\frac{X_{\varepsilon}-\xi_{\alpha,\varepsilon}}{\vert X_{\varepsilon}-\xi_{\alpha,\varepsilon}\vert}\cdot\frac{V_{\varepsilon}-\eta_{\alpha,\varepsilon}}{\varepsilon}+\frac{X_{\varepsilon}-\xi_{\alpha,\varepsilon}}{\vert X_{\varepsilon}-\xi_{\alpha,\varepsilon}\vert^2}\cdot\frac{\eta_{\alpha,\varepsilon}}{\varepsilon}\right).
\end{align*}
The lemma follows from the fact $\vert\eta_{\alpha,\varepsilon}(s)\vert\le C$ by \eqref{eq-priori-1} with $C$ depending only on $K_0,K_1$.
\end{proof}

\begin{proof}[Proof of Proposition~\ref{prn-estimate-Lk}]
We split the integral domains in $L_{k,\varepsilon}(t)$ into two parts to obtain
\begin{align*}
&\int_0^t\iint_{\mathbb{R}^2\times\mathbb{R}^2}\frac{\mathbf{h}_{\varepsilon}(s)^{k/2}f_{\varepsilon}^0}{\vert X_{\varepsilon}(s)-\xi_{\alpha,\varepsilon}(s)\vert}\,\mathrm{d}x\,\mathrm{d}v\,\mathrm{d}s\\
&\le\int_0^t\iint_{\vert X_{\varepsilon}(s)-\xi_{\alpha,\varepsilon}(s)\vert\le\frac{M}{2}}+\int_0^t\iint_{\vert X_{\varepsilon}(s)-\xi_{\alpha,\varepsilon}(s)\vert>\frac{M}{2}}\frac{\mathbf{h}_{\varepsilon}(s)^{k/2}f_{\varepsilon}^0}{\vert X_{\varepsilon}(s)-\xi_{\alpha,\varepsilon}(s)\vert}\,\mathrm{d}x\,\mathrm{d}v\,\mathrm{d}s\\
&\le\int_0^t\iint_{\vert X_{\varepsilon}(s)-\xi_{\alpha,\varepsilon}(s)\vert\le\frac{M}{2}}\frac{\mathbf{h}_{\varepsilon}(s)^{k/2}f_{\varepsilon}^0}{\vert X_{\varepsilon}(s)-\xi_{\alpha,\varepsilon}(s)\vert}\,\mathrm{d}x\,\mathrm{d}v\,\mathrm{d}s+\frac{2}{M}H_{k,\varepsilon}(t)t.
\end{align*}

By Lemma~\ref{lem-estimate-character}, we have
\begin{align*}
&\int_0^t\iint_{\vert X_{\varepsilon}(s)-\xi_{\alpha,\varepsilon}(s)\vert\le\frac{M}{2}}\frac{\mathbf{h}_{\varepsilon}(s)^{k/2}f_{\varepsilon}^0}{\vert X_{\varepsilon}(s)-\xi_{\alpha,\varepsilon}(s)\vert}\,\mathrm{d}x\,\mathrm{d}v\,\mathrm{d}s\\
&\le\int_0^t\iint_{\mathbb{R}^2\times\mathbb{R}^2}\mathbf{h}_{\varepsilon}(s)^{k/2}f_{\varepsilon}^0\Big(\varepsilon^2\frac{\mathrm{d}^{2}}{\mathrm{d}s^{2}}\vert X_{\varepsilon}(s)-\xi_{\alpha,\varepsilon}(s)\vert+\varepsilon^{-1}\vert V_{\varepsilon}(s)-\eta_{\alpha,\varepsilon}(s)\vert\\
&\qquad\qquad\qquad\qquad\qquad+\vert E_{\varepsilon}(s,X_{\varepsilon}(s))\vert+\vert E_{\varepsilon}(s,\xi_{\alpha,\varepsilon}(s))\vert+\frac{3N}{M}\Big)\,\mathrm{d}x\,\mathrm{d}v\,\mathrm{d}s\\
&=:I_{\varepsilon}^1(t)+I_{\varepsilon}^2(t)+I_{\varepsilon}^3(t)+I_{\varepsilon}^4(t)+\frac{3N}{M}H_{k,\varepsilon}(t)t.
\end{align*}
Now we estimate each term above. 

{\bf Estimate on $I_{\varepsilon}^1(t)$:} Firstly, using integration by parts in time, we obtain
\begin{align*}
I_{\varepsilon}^1(s)&=\varepsilon^2\iint_{\mathbb{R}^2\times\mathbb{R}^2} f_{\varepsilon}^0\Big(\int_0^t\mathbf{h}_{\varepsilon}(s)^{k/2}\frac{\mathrm{d}^2}{\mathrm{d}s^2}\vert X_{\varepsilon}(s)-\xi_{\alpha,\varepsilon}(s)\vert \,\mathrm{d}s\Big)\,\mathrm{d}x\,\mathrm{d}v\\
&=\varepsilon^2\iint_{\mathbb{R}^2\times\mathbb{R}^2} f_{\varepsilon}^0\Big\{\Big[\mathbf{h}_{\varepsilon}(s)^{k/2}\frac{\mathrm{d}}{\mathrm{d}s}\vert X_{\varepsilon}(s)-\xi_{\alpha,\varepsilon}(s)\vert \Big]\Big\vert^t_0\\
&\qquad\qquad\qquad-\int_0^t\frac{\mathrm{d}}{\mathrm{d}s}\mathbf{h}_{\varepsilon}(s)^{k/2}\frac{\mathrm{d}}{\mathrm{d}s}\vert X_{\varepsilon}(s)-\xi_{\alpha,\varepsilon}(s)\vert \,\mathrm{d}s\Big\}\,\mathrm{d}x\,\mathrm{d}v,
\end{align*}
by direct calculation, we have
\begin{equation*}
\frac{\mathrm{d}}{\mathrm{d}s}\vert X_{\varepsilon}(s)-\xi_{\alpha,\varepsilon}(s)\vert=\frac{X_{\varepsilon}(s)-\xi_{\alpha,\varepsilon}(s)}{\vert X_{\varepsilon}(s)-\xi_{\alpha,\varepsilon}(s)\vert }\cdot\frac{V_{\varepsilon}(s)-\eta_{\alpha,\varepsilon}(s)}{\varepsilon}\le C\varepsilon^{-1}\mathbf{h}_{\varepsilon}(s)^{1/2},
\end{equation*}
combining with Lemma~\ref{lem-estimate-pointwise-energy}, we have
\begin{equation*}
\Big[\mathbf{h}_{\varepsilon}(s)^{k/2}\frac{\mathrm{d}}{\mathrm{d}s}\vert X_{\varepsilon}(s)-\xi_{\alpha,\varepsilon}(s)\vert
\Big]\Big\vert^t_0\le C\varepsilon^{-1}(\mathbf{h}_{\varepsilon}(t)^{(k+1)/2}+\mathbf{h}_{\varepsilon}(0)^{(k+1)/2})
\end{equation*}
and
\begin{align*}
&\int_0^t\frac{\mathrm{d}}{\mathrm{d}s}\mathbf{h}_{\varepsilon}(s)^{k/2}\frac{\mathrm{d}}{\mathrm{d}s}\vert X_{\varepsilon}(s)-\xi_{\alpha,\varepsilon}(s)\vert \,\mathrm{d}s\\
&\le C\varepsilon^{-2}\int_0^t\mathbf{h}_{\varepsilon}(s)^{\frac{k-1}{2}}
\left(\sqrt{\mathbf{h}_{\varepsilon}(s)}\vert E_{\varepsilon}(s,X_{\varepsilon}(s))\vert+\sum_{\alpha}\frac{1}{\vert X_{\varepsilon}(s)-\xi_{\alpha,\varepsilon}(s)\vert}\right)\,\mathrm{d}s\\
&\le C\varepsilon^{-2}\int_0^t\mathbf{h}_{\varepsilon}(s)^{\frac{k}{2}}\vert E_{\varepsilon}(s,X_{\varepsilon}(s))\vert+\sum_{\alpha}\frac{\mathbf{h}_{\varepsilon}(s)^{\frac{k-1}{2}}
}{\vert X_{\varepsilon}(s)-\xi_{\alpha,\varepsilon}(s)\vert}\,\mathrm{d}s.
\end{align*}
Hence we have
\begin{align*}
I_{\varepsilon}^1(s)&\le C\varepsilon\iint_{\mathbb{R}^2\times\mathbb{R}^2} f_{\varepsilon}^0\Big(\mathbf{h}_{\varepsilon}(t)^{(k+1)/2}+\mathbf{h}_{\varepsilon}(0)^{(k+1)/2}\Big)\,\mathrm{d}x\,\mathrm{d}v\\
&\quad+C\iint_{\mathbb{R}^2\times\mathbb{R}^2} f_{\varepsilon}^0\left\{\int_0^t\mathbf{h}_{\varepsilon}(s)^{\frac{k}{2}}\vert E_{\varepsilon}(s,X_{\varepsilon}(s))\vert+\sum_{\alpha}\frac{\mathbf{h}_{\varepsilon}(s)^{\frac{k-1}{2}}}{\vert X_{\varepsilon}(s)-\xi_{\alpha,\varepsilon}(s)\vert}\,\mathrm{d}s\right\}\,\mathrm{d}x\,\mathrm{d}v\\
&\le C\Big(\varepsilon H_{k+1,\varepsilon}(t)+\int_0^t\iint_{\mathbb{R}^2\times\mathbb{R}^2} \mathbf{h}(s)^{k/2}f_{\varepsilon}^0\vert E_{\varepsilon}(s,X_{\varepsilon}(s))\vert \,\mathrm{d}x\,\mathrm{d}v\,\mathrm{d}s\\
&\quad+\sum_{\alpha}\int_0^t\iint_{\mathbb{R}^2\times\mathbb{R}^2}\frac{\mathbf{h}_{\varepsilon}(s)^{(k-1)/2}f_{\varepsilon}^0}{\vert X_{\varepsilon}(s)-\xi_{\alpha,\varepsilon}(s)\vert}\,\mathrm{d}x\,\mathrm{d}v\,\mathrm{d}s\Big)\\
&\le C\left(\varepsilon H_{k+1,\varepsilon}(t)+\|f_{\varepsilon}^0\|_{\infty}^{\frac{l-k}{l+2}}H_{l,\varepsilon}(t)^{\frac{k+2}{l+2}}\int_0^t\|E_{\varepsilon}(s)\|_{\frac{l+2}{l-k}}\,\mathrm{d}s+L_{k-1,\varepsilon}(t)\right),
\end{align*}
where we have used the H\"older inequality and Lemma~\ref{lem-interpo-Hk}.

{\bf Estimate on $I_{\varepsilon}^2(t)$:} By definition, we have $\vert V_{\varepsilon}(s)-\eta_{\alpha,\varepsilon}(s)\vert\le C\sqrt{\mathbf{h}_{\varepsilon}(s)}$. Hence
\begin{equation*}
I_{\varepsilon}^2=\varepsilon^{-1}\int_0^t\iint_{\mathbb{R}^2\times\mathbb{R}^2} \mathbf{h}_{\varepsilon}(s)^{k/2}f_{\varepsilon}^0\vert V_{\varepsilon}(s)-\eta_{\alpha,\varepsilon}(s)\vert\,\mathrm{d}x\,\mathrm{d}v\,\mathrm{d}s\le C\varepsilon^{-1}H_{k+1,\varepsilon}(t)t.
\end{equation*}

{\bf Estimate on $I_{\varepsilon}^3(s)$:} By  H\"older inequality and Lemma~\ref{lem-interpo-Hk} we have
\begin{align*}
I_{\varepsilon}^3&=\int_0^t\iint_{\mathbb{R}^2\times\mathbb{R}^2} \mathbf{h}_{\varepsilon}(s)^{k/2}f_{\varepsilon}^0\vert E_{\varepsilon}(s,X_{\varepsilon}(s))\vert\,\mathrm{d}x\,\mathrm{d}v\,\mathrm{d}s\\
&\le\int_0^t\Big\|\int_{\mathbb{R}^2}h_{\varepsilon}^{k/2}f_{\varepsilon}(t,x,v)\,\mathrm{d}v\Big\|_{\frac{l+2}{k+2}}\|E_{\varepsilon}(s)\|_{\frac{l+2}{l-k}}\,\mathrm{d}s\\
&\le C\|f_{\varepsilon}^0\|_{\infty}^{\frac{l-k}{l+2}}H_{l,\varepsilon}(t)^{\frac{k+2}{l+2}}\int_0^t\|E_{\varepsilon}(s)\|_{\frac{l+2}{l-k}}\,\mathrm{d}s.
\end{align*}

{\bf Estimate on $I_{\varepsilon}^4(t)$:}  We have
\begin{align*}
I_{\varepsilon}^4&=\int_0^t\iint_{\mathbb{R}^2\times\mathbb{R}^2} \mathbf{h}_{\varepsilon}(s)^{k/2}f_{\varepsilon}^0\vert E_{\varepsilon}(s,\xi_{\alpha,\varepsilon}(s))\vert\,\mathrm{d}x\,\mathrm{d}v\,\mathrm{d}s\\
&\le CH_{k,\varepsilon}(t)\int_0^t\vert E_{\varepsilon}(s,\xi_{\alpha,\varepsilon}(s))\vert \,\mathrm{d}s\le CH_{k,\varepsilon}(t)L_{0,\varepsilon}(t).
\end{align*}

Combining the estimates on $I_{\varepsilon}^i(t)$ for $i=1,2,3,4$, we have
\begin{align*}
&\int_0^t\iint_{\mathbb{R}^2\times\mathbb{R}^2}\frac{\mathbf{h}_{\varepsilon}(s)^{k/2}f_{\varepsilon}^0}{\vert X_{\varepsilon}(s)-\xi_{\alpha,\varepsilon}(s)\vert}\,\mathrm{d}x\,\mathrm{d}v\,\mathrm{d}s\\
&\le C\left(\varepsilon H_{k+1,\varepsilon}(t)+\|f_{\varepsilon}^0\|_{\infty}^{\frac{l-k}{l+2}}H_{l,\varepsilon}(t)^{\frac{k+2}{l+2}}\int_0^t\|E_{\varepsilon}(s)\|_{\frac{l+2}{l-k}}\,\mathrm{d}s+L_{k-1,\varepsilon}(t)\right)\\
&\qquad+C\varepsilon^{-1}H_{k+1,\varepsilon}(t)t+CH_{k,\varepsilon}(t)L_{0,\varepsilon}(t)+\frac{3N+2}{M}H_{k,\varepsilon}(t)t.
\end{align*}
Summing the above inequalities with index $\alpha$, the conclusion follows.
\end{proof}

\section*{Acknowledgements}

\bibliographystyle{abbrv}
{\footnotesize\bibliography{ref}}

\begin{thebibliography}{10}

\bibitem{AW21}
A.~Arroyo-Rabasa and R.~Winter.
\newblock Debye screening for the stationary {V}lasov-{P}oisson equation in
  interaction with a point charge.
\newblock {\em Comm. Partial Differential Equations}, 46(8):1569--1584, 2021.

\bibitem{Ars73}
A.~A. Arsenev.
\newblock Existence in the large of a weak solution of {V}lasov's system of
  equations.
\newblock {\em Dokl. Akad. Nauk SSSR}, 213:761--763, 1973.

\bibitem{BOS09}
M.~Bostan.
\newblock The {V}lasov-{P}oisson system with strong external magnetic field.
  {F}inite {L}armor radius regime.
\newblock {\em Asymptot. Anal.}, 61(2):91--123, 2009.

\bibitem{Bre00}
Y.~Brenier.
\newblock Convergence of the {V}lasov-{P}oisson system to the incompressible
  {E}uler equations.
\newblock {\em Comm. Partial Differential Equations}, 25(3-4):737--754, 2000.

\bibitem{CM10}
S.~Caprino and C.~Marchioro.
\newblock On the plasma-charge model.
\newblock {\em Kinet. Relat. Models}, 3(2):241--254, 2010.

\bibitem{CMMP12}
S.~Caprino, C.~Marchioro, E.~Miot, and M.~Pulvirenti.
\newblock On the attractive plasma-charge system in 2-d.
\newblock {\em Comm. Partial Differential Equations}, 37(7):1237--1272, 2012.

\bibitem{Del91}
J.-M. Delort.
\newblock Existence de nappes de tourbillon en dimension deux.
\newblock {\em J. Amer. Math. Soc.}, 4(3):553--586, 1991.

\bibitem{DMS15}
L.~Desvillettes, E.~Miot, and C.~Saffirio.
\newblock Polynomial propagation of moments and global existence for a
  {V}lasov-{P}oisson system with a point charge.
\newblock {\em Ann. Inst. H. Poincar\'{e} Anal. Non Lin\'{e}aire},
  32(2):373--400, 2015.

\bibitem{DL89ODE}
R.~J. DiPerna and P.-L. Lions.
\newblock Ordinary differential equations, transport theory and {S}obolev
  spaces.
\newblock {\em Invent. Math.}, 98(3):511--547, 1989.

\bibitem{DG23}
M.~Donati and L.~Godard-Cadillac.
\newblock H\"{o}lder regularity for collapses of point-vortices.
\newblock {\em Nonlinearity}, 36(11):5773--5818, 2023.

\bibitem{FR16}
F.~Filbet and L.~M. Rodrigues.
\newblock Asymptotically stable particle-in-cell methods for the
  {V}lasov-{P}oisson system with a strong external magnetic field.
\newblock {\em SIAM J. Numer. Anal.}, 54(2):1120--1146, 2016.

\bibitem{FS98}
E.~Fr\'{e}nod and E.~Sonnendr\"{u}cker.
\newblock Homogenization of the {V}lasov equation and of the {V}lasov-{P}oisson
  system with a strong external magnetic field.
\newblock {\em Asymptot. Anal.}, 18(3-4):193--213, 1998.

\bibitem{FS00}
E.~Fr\'{e}nod and E.~Sonnendr\"{u}cker.
\newblock Long time behavior of the two-dimensional {V}lasov equation with a
  strong external magnetic field.
\newblock {\em Math. Models Methods Appl. Sci.}, 10(4):539--553, 2000.

\bibitem{FS01}
E.~Fr\'{e}nod and E.~Sonnendr\"{u}cker.
\newblock The finite {L}armor radius approximation.
\newblock {\em SIAM J. Math. Anal.}, 32(6):1227--1247, 2001.

\bibitem{GLS14}
O.~Glass, C.~Lacave, and F.~Sueur.
\newblock On the motion of a small body immersed in a two-dimensional
  incompressible perfect fluid.
\newblock {\em Bull. Soc. Math. France}, 142(3):489--536, 2014.

\bibitem{GLS16}
O.~Glass, C.~Lacave, and F.~Sueur.
\newblock On the motion of a small light body immersed in a two dimensional
  incompressible perfect fluid with vorticity.
\newblock {\em Comm. Math. Phys.}, 341(3):1015--1065, 2016.

\bibitem{GMS18}
O.~Glass, A.~Munnier, and F.~Sueur.
\newblock Point vortex dynamics as zero-radius limit of the motion of a rigid
  body in an irrotational fluid.
\newblock {\em Invent. Math.}, 214(1):171--287, 2018.

\bibitem{GS19}
O.~Glass and F.~Sueur.
\newblock Dynamics of several rigid bodies in a two-dimensional ideal fluid and
  convergence to vortex systems.
\newblock {\em arXiv: Analysis of PDEs}, 2019.

\bibitem{GS99}
F.~Golse and L.~Saint-Raymond.
\newblock The {V}lasov-{P}oisson system with strong magnetic field.
\newblock {\em J. Math. Pures Appl. (9)}, 78(8):791--817, 1999.

\bibitem{GS03}
F.~Golse and L.~Saint-Raymond.
\newblock The {V}lasov-{P}oisson system with strong magnetic field in
  quasineutral regime.
\newblock {\em Math. Models Methods Appl. Sci.}, 13(5):661--714, 2003.

\bibitem{GNPS05}
T.~Goudon, J.~Nieto, F.~Poupaud, and J.~Soler.
\newblock Multidimensional high-field limit of the electrostatic
  {V}lasov-{P}oisson-{F}okker-{P}lanck system.
\newblock {\em J. Differential Equations}, 213(2):418--442, 2005.

\bibitem{HW22}
R.~M. H\"ofer and R.~Winter.
\newblock A fast point charge interacting with the screened vlasov-poisson
  system, 2022 (Soon to be published in Analysis \& PDE).

\bibitem{LM09}
C.~Lacave and E.~Miot.
\newblock Uniqueness for the vortex-wave system when the vorticity is constant
  near the point vortex.
\newblock {\em SIAM J. Math. Anal.}, 41(3):1138--1163, 2009.

\bibitem{LM21}
C.~Lacave and E.~Miot.
\newblock The vortex-wave system with gyroscopic effects.
\newblock {\em Ann. Sc. Norm. Super. Pisa Cl. Sci. (5)}, 22(1):1--30, 2021.

\bibitem{LP91}
P.-L. Lions and B.~Perthame.
\newblock Propagation of moments and regularity for the {$3$}-dimensional
  {V}lasov-{P}oisson system.
\newblock {\em Invent. Math.}, 105(2):415--430, 1991.

\bibitem{Loe06}
G.~Loeper.
\newblock Uniqueness of the solution to the {V}lasov-{P}oisson system with
  bounded density.
\newblock {\em J. Math. Pures Appl. (9)}, 86(1):68--79, 2006.

\bibitem{MMP11}
C.~Marchioro, E.~Miot, and M.~Pulvirenti.
\newblock The {C}auchy problem for the 3-{D} {V}lasov-{P}oisson system with
  point charges.
\newblock {\em Arch. Ration. Mech. Anal.}, 201(1):1--26, 2011.

\bibitem{MP91}
C.~Marchioro and M.~Pulvirenti.
\newblock On the vortex-wave system.
\newblock In {\em Mechanics, analysis and geometry: 200 years after
  {L}agrange}, North-Holland Delta Ser., pages 79--95. North-Holland,
  Amsterdam, 1991.

\bibitem{Mio16}
E.~Miot.
\newblock On the gyrokinetic limit for the two-dimensional {V}lasov-{P}oisson
  system.
\newblock {\em arXiv:1603.04502v1}, 2016.

\bibitem{Mio16CMP}
E.~Miot.
\newblock A uniqueness criterion for unbounded solutions to the
  {V}lasov-{P}oisson system.
\newblock {\em Comm. Math. Phys.}, 346(2):469--482, 2016.

\bibitem{Mio19}
E.~Miot.
\newblock The gyrokinetic limit for the {V}lasov-{P}oisson system with a point
  charge.
\newblock {\em Nonlinearity}, 32(2):654--677, 2019.

\bibitem{PW21}
B.~Pausader and K.~Widmayer.
\newblock Stability of a point charge for the {V}lasov-{P}oisson system: the
  radial case.
\newblock {\em Comm. Math. Phys.}, 385(3):1741--1769, 2021.

\bibitem{PWY22}
B.~Pausader, K.~Widmayer, and J.~Yang.
\newblock Stability of a point charge for the repulsive vlasov-poisson system.
\newblock 2022. arXiv:2207.05644.

\bibitem{Pfa92}
K.~Pfaffelmoser.
\newblock Global classical solutions of the {V}lasov-{P}oisson system in three
  dimensions for general initial data.
\newblock {\em J. Differential Equations}, 95(2):281--303, 1992.

\bibitem{Pou02}
F.~Poupaud.
\newblock Diagonal defect measures, adhesion dynamics and {E}uler equation.
\newblock {\em Methods Appl. Anal.}, 9(4):533--561, 2002.

\bibitem{Sai00}
L.~Saint-Raymond.
\newblock The gyrokinetic approximation for the {V}lasov-{P}oisson system.
\newblock {\em Math. Models Methods Appl. Sci.}, 10(9):1305--1332, 2000.

\bibitem{Sai02}
L.~Saint-Raymond.
\newblock Control of large velocities in the two-dimensional gyrokinetic
  approximation.
\newblock {\em J. Math. Pures Appl. (9)}, 81(4):379--399, 2002.

\bibitem{Sch95}
S.~Schochet.
\newblock The weak vorticity formulation of the {$2$}-{D} {E}uler equations and
  concentration-cancellation.
\newblock {\em Comm. Partial Differential Equations}, 20(5-6):1077--1104, 1995.

\bibitem{UO78}
S.~Ukai and T.~Okabe.
\newblock On classical solutions in the large in time of two-dimensional
  {V}lasov's equation.
\newblock {\em Osaka Math. J.}, 15(2):245--261, 1978.

\bibitem{WZ23VPenergy}
J.~Wu and X.~Zhang.
\newblock The energy conservation of {V}lasov-{P}oisson systems.
\newblock {\em Acta Math. Sci. Ser. B (Engl. Ed.)}, 43(2):668--674, 2023.

\bibitem{WZ23}
J.~Wu and X.~Zhang.
\newblock Moment propagation of the plasma-charge model with a time-varying
  magnetic field.
\newblock {\em J.Stat.Phys.}, 190(183), 2023.

\bibitem{XZ21}
H.~Xiong and X.~Zhang.
\newblock Time evolution of a plasma-charge system with infinite mass and
  velocities.
\newblock {\em J. Differential Equations}, 278:1--49, 2021.

\end{thebibliography}
\end{document}